\documentclass[12 pt]{amsart}

\usepackage[dvips]{color}
\usepackage{graphicx}
\usepackage{epsfig}
\usepackage{tikz}
\usepackage{enumerate}
\usepackage{enumitem}
\usepackage{color}
\usepackage{latexsym}
\usepackage{amssymb}
\usepackage{amsmath}
\usepackage{amsthm}
\usepackage{amscd}
\usepackage{latexsym}
\usepackage{amssymb}
\usepackage{amsmath}
\usepackage{amsthm}
\usepackage{amscd}
\theoremstyle{plain}
\newtheorem{thm}{Theorem}[section]
\newtheorem{lem}[thm]{Lemma}
\newtheorem{defin}[thm]{Definition}
\newtheorem{prop}[thm]{Proposition}
\theoremstyle{definition}
\newtheorem{cor}[thm]{Corollary}
\newtheorem{rem}[thm]{Remark}

\newtheorem{ques}[thm]{Question}
\newcommand{\LipP}{Lip_1(X(P))}

\newcommand{\CalS}{\mathcal{S}}

\newcommand{\FF}{\hat{\mathcal{F}}}

\newcommand{\supp}{\text{supp}}

\newcommand{\future}[1]{{}}  

\begin{document}

\title{Weakly Mixing Polygonal Billiards}
\author{Jon Chaika} 
\author{Giovanni Forni}

\address{Department  of Mathematics\\
  University of Utah \\
  Salt Lake City, UT USA}

\address{Department  of Mathematics\\
  University of Maryland \\
  College Park, MD USA}
  
\email
    {chaika@math.utah.edu}
  \email
    {gforni@umd.edu}
    
    
\begin{abstract}   We prove that  there exists a $G_\delta$ dense set of (non-rational) polygons such the billiard flow is weakly mixing with respect to the Liouville measure (on the unit tangent bundle to the billiard). This follows, via a Baire category argument, from showing that for any translation surface the product of the flows in almost every pair of directions is ergodic with respect to Lebesgue measure. This in turn is proven by showing that for every translation surface the flows in almost every pair of directions do not share non-trivial common eigenvalues. \end{abstract} 

\maketitle

\section{Introduction}

A basic question in ergodic theory is the dynamical properties of the billiard flow in various (planar) 
 domains (which we will call ``tables'') \cite{KaBil},   \cite{MT}, \cite{Tab}, \cite{CM}, \cite{Gut}. Frequently, the table is assumed to have piecewise $C^1$ boundary and the billiard is assumed to be a massless point traveling without friction on the interior of the table which experiences elastic collision when it hits the boundary of the table.  We will be concerned with such a system where the table is a polygon. The main result of this paper is,

\begin{thm}\label{thm:billiard} There exists a weakly mixing billiard flow in a polygon.
\end{thm}
In fact, we answer a conjecture stated by E.~Gutkin and A.~Katok (see \cite{GK}, \S1)  that the set of polygonal tables with a weakly mixing billiard flow is a dense $G_\delta$ subset of  the appropriate space of all polygonal tables with a fixed number of vertices.

Let $P$ denote a polygon. The billiard flow is a flow $F^t:(P \times S^1)/\sim \,\to\, (P \times S^1)/\sim $, where $(p,\theta) \sim (q,\psi)$ if $p=q \in \partial P$ and  the angle between $\theta$ and the side of $P$ at $p$ is $\pi$ minus the angle between $\psi$ and the side of $P$ at $p$. This flow is defined for any orbit that does not orbit into the vertices of $P$. This flow preserves  the Lebesgue measure on $P\times S^1$, defined as the product of the  Lebesgue measure on $P$ times the Lebesgue measure on $S^1$ (pushed forward by $\sim$). 

A natural dynamical consequence of this result is  that there exists a polygon $P$ of area $1$ so that for any rectangles $R, R' \subset P$ and for any intervals $I$, $I'\subset S^1$,  for almost every $p, p' \in P$ and $\theta,\psi\in S^1$ 
we have \begin{multline*}\label{eq:conseq}
\frac{|\{0\leq t\leq T:F^t(p,\theta)\in R \times I  \text{ and }  F^t(p',\psi) \in R'\times I'\}|}{T} \rightarrow \\ |R||I|\cdot |R'|  |I'| .
\end{multline*}
 
S.~Kerckhoff, H.~Masur and J.~Smillie proved that there were ergodic billiard flows in polygons~\cite{KMS}. The significant general results about billiard flows in polygons are C.~Boldrighini, M.~Keane and F.~Marchetti's result that they have at most a countable set of directions containing periodic trajectories \cite{BKM} and A.~Katok's result that the billiard flow has zero entropy \cite[Section 3]{Ka87}. By a result of A.~Katok \cite[Section 3]{Ka87}  the zero entropy property implies a subexponential bound on the growth of complexity, in particular for the counting function of generalized diagonals and periodic orbits. Recently, D.~Scheglov \cite{Sch13}, \cite{Sch} improved this bound in the case of almost all triangles to ``weakly exponential''. We recall that H.~Masur~\cite{Ms90} showed the counting function of generalized diagonals grows quadratically for rational tables, and it has been conjectured that polynomial bounds should hold for general typical polygons (\cite{Ka87}, \S4).   
The most significant recent results on the ergodic properties of billiards in polygons, all of which are in different contexts, are the results of 
A.~Avila and V.~Delecroix \cite{AD} and D.~Aulicino, A.~Avila and V.~Delecroix~\cite{AAD}, which prove weak mixing {\it in almost all directions} in certain {\it rational polygons} and the proof by A.~M\'alaga and S.~Troubetzkoy \cite{MaTr} of the weak mixing property in almost all directions for rational billiards in generic {\it polygons with vertical - horizontal sides}.  In these cases the unit tangent bundle splits into invariant surfaces, and the weak mixing property is proved with respect to the two dimensional measure on the generic invariant surface, as opposed to our setting in which we prove weak mixing 
 with respect to the ($3$-dimensional) Liouville measure   on the unit tangent bundle. Another recent result is the proof by J.~Bobok and S.~Troubetzkoy \cite{BoTr} of {\it topological  weak mixing} for the billiard map  (that is, for  the first return map to the sides of the polygon, which is a $\mathbb{Z}$-action, as opposed to the  billiard flows we consider) in the generic polygon.

Similar to the proof of existence of ergodic billiards \cite{KMS}, and of the {\it directional} weak mixing results of \cite{GK}, our result is obtained from a result about every translation surface via a Baire category argument. A translation surface is a pair $(X,\omega)$ where $X$ is a Riemann surface and $\omega$ is an Abelian differential. 
From $\omega$ we obtain an $S^1$ family of vector fields on $X$, which in turn give flows $F_{r_\theta \omega}^t$ on $X$ and a Lebesgue measure 
on $X$. We denote  this normalized Lebesgue measure $\lambda^2$, regardless of the surface, which is preserved by these flows. 

\begin{thm}\label{thm:erg prod} For a.e. $(\theta, \phi) \in S^1\times S^1$ the flow $F_{r_\theta \omega}^t \times  F_{r_\phi \omega}^t $ is $\lambda^2 \times \lambda^2$ ergodic.
\end{thm}
P.~Hubert and the first named author \cite{CH} previously showed that for almost every surface (with respect to any $SL(2,\mathbb{R})$-invariant measure) the product of the flow in almost every direction is uniquely ergodic. 
Our methods for establishing Theorem \ref{thm:erg prod} are spectral, as we show the following, which because the straight line flow on any translation surface is $\lambda^2$ ergodic  in almost every direction \cite{KMS}, is well known to imply Theorem~\ref{thm:erg prod}   (see for instance \cite{KT}, Prop. 4.2):
\begin{thm}\label{thm:main}For every $\alpha\neq 0 $ and every translation surface $\omega$ we have $|\{\theta \in S^1:F_{r_\theta \omega}^t \text{ has eigenvalue } \alpha  \}|=0$\,.
\end{thm}
A.~Avila and the second named author \cite{AF} showed that for almost every translation surface (in genus at least $2$) the flow in almost every direction is weakly mixing, which implies the above result for almost every surface. In fact, by an announcement of D.~Aulicino, A.~Avila and V.~Delecroix \cite{AAD}, Theorem \ref{thm:main} holds for almost every surface with respect to every $SL(2,\mathbb{R})$ invariant measure. (They show that for any $SL(2,\mathbb{R})$ ergodic measure, which is not supported on branched covers of tori, the vertical flow on almost every surface is weakly mixing. This gives Theorem \ref{thm:erg prod} for these measures.
By an argument based on rigidity sequences, one can show that Theorem~\ref{thm:main} holds for any branched cover of a torus. In fact, one can prove, see for instance \cite{FH19}, that for any $\alpha\in \mathbb R\setminus \{0\}$ and for the flow in almost every direction on the branched cover of the torus, there exists a rigidity sequence $t_j$  so that $e^{2\pi \imath \alpha t_j}$ does not converge to $1$). 

Note, it is well known that single cylinder surfaces give many examples of surfaces (even with full orbit closure)  where every direction has 
a non-constant eigenfunction.

\subsection{Organization of the paper}

In Section~\ref{subsec:open} below we gather some fundamental open questions in the  ergodic theory
of billiards in polygons. In Section~\ref{sec:background} we recall some basic material about the Teichm\"uller  flow and the renormalization cocycle for translation flows, the Kontsevich--Zorich cocycle over the Teichm\"uller flow on the Hodge bundle. We conclude the section, in Section~\ref{subsec:outline}, with an outline of our argument. In Section~\ref{sec:grow} we derive several consequences of the work of  \cite{EskMir}, \cite{EskMirMo} and \cite{CE},  including results  on the growth of vectors under the action of the Kontsevich--Zorich cocycle   and on averages along horocycle arcs of the pushforward of ``height'' functions (functions as in Theorem \ref{thm:EMM 2.13})  for the moduli space by the geodesic flow. 
Section \ref{sec:lip and stuff} contains, in Section~\ref{subsec:curve_grow}, preparatory results on the growth of curves in the Hodge bundle, which are derived from results of the previous section, and, in Section~\ref{subsec:largedev}, some standard large deviation results, which are included for convenience of the reader. 
Section~\ref{sec:key prop} contains the core of the argument, that is, the key proposition on the controlled growth of curves under
the Kontsevich-Zorich cocycle (Prop.~\ref{prop:better}).  In Section~\ref{sec:main proof} after recalling Veech's criterion for weak mixing, we prove our main result, Theorem~\ref{thm:main}. Finally, in Section~\ref{sec:billiard} we derive our result on the existence of weakly mixing polygonal billiards.

\subsection{Open questions}
\label{subsec:open} 

\begin{ques} Does every translation surface $\omega$ where the flow is not weakly mixing in almost every direction have the property that the line $\mathbb R [\Im(r_\theta\omega)]$ spanned by the cohomology class $[\Im(r_\theta\omega)]$ of the imaginary part of its rotation $r_\theta\omega$ intersects an integer translate of the stable subspace of the $SL(2,\mathbb{R})$ subbundle in a set of directions $\theta\in S^1$ of  positive measure?
\end{ques}
\begin{ques} For every translation surface is the flow in almost every pair of directions uniquely ergodic? Spectrally singular (modulo constants)?
\end{ques}

There are also many natural questions about the flow of billiards in irrational polygons.
\begin{ques}Is there a polygon with a mixing billiard flow? Topologically mixing? Minimal? Is the billiard flow in every irrational polygon ergodic? Weak mixing? Mixing? 
\end{ques}
Note that numerical experiments~\cite{CP} suggest that every billiard in an irrational triangle is mixing.

\subsection{Acknowledgments}
The first named author  thanks DMS-1452762, a Poincar\'{e} chair, a Sloan fellowship and a Warnock chair for support. The first named author also thanks the University of Maryland for its hospitality. 
The second named author  was supported by the NSF grant DMS 1600687 and by a Research Chair of the Fondation Sciences Math\'ematiques de Paris (FSMP). He would also like to thank the University of Utah for its hospitality. 
The authors would like to thank CIRM, IHP, University of Zurich, ETH Zurich and Oberwolfach for their hospitality during work on this project. The authors thank the referee for helpful comments which improved the paper.

\section{Background}
\label{sec:background}
As stated earlier, a translation surface is a pair $(X,\omega)$ where $X$ is a finite type Riemann surface and $\omega$ is an Abelian differential. 
The space of translation surfaces is stratified by the order and number of the zeros of the Abelian differential. Let $\mathcal{H}(\alpha)$ be the moduli space of translation surfaces 
$(X, \omega)$ such that $\omega$ has $k$ zeros of orders $\alpha:=(\alpha_1, \dots, \alpha_k)$, which we call a stratum. The (real) Hodge bundle of the stratum is a bundle over each stratum whose fiber at $(X,\omega)$ is 
$H^1(X;\mathbb{R})$. 
Given a translation surface $(X,\omega)$, integrating $\omega$ provides charts for $X \setminus \Sigma$ where $\Sigma$ is the set of zeros of $\omega$. From its action on these charts, the group $SL(2,\mathbb{R})$ of real matrices with determinant one acts on each translation surface, preserving the stratum it's in. 

The group $SL(2,\mathbb{R})$ also acts on the Hodge bundle since it acts on the base $\mathcal{H}(\alpha)$ of the bundle
and this action can be lifted to the bundle by parallel transport of cohomology classes with respect to a natural
flat connection. 

The Kontsevich-Zorich cocycle is given  by the action of $SL(2,\mathbb{R})$  on the Hodge bundle. It is a cocycle over
the action of the group $SL(2,\mathbb{R})$ on the stratum  $\mathcal{H}(\alpha)$.  
See the lecture notes of the second named author and Carlos Matheus~\cite{FM}  for a detailed description of the above material.

Let $$h_s=\begin{pmatrix}1&s\\0&1\end{pmatrix}, \, \hat{h}_s=\begin{pmatrix}1&0\\s&1\end{pmatrix}, \, 
 g_t=\begin{pmatrix}e^{t}&0\\0&e^{-t}\end{pmatrix},$$ 
denote, respectively, the Teichm\"uller horocycle flows and the  Teichm\"uller geodesic flow. Also let
$$
r_\theta=\begin{pmatrix}\cos(\theta)&-\sin(\theta)\\ \sin(\theta)&\cos(\theta)\end{pmatrix}
$$
denote the group of rotations. 

Let $I\subset \mathbb {R}$ be any interval. Given a horocycle arc 
$\{ h_s \omega \vert s \in I \}$ at $\omega \in \mathcal H(\alpha)$, a horocycle section at $\omega$ is a map  $\phi: I \to H^1(M, \mathbb{R})$ such that $\phi(s) \in H^1_{h_s\omega} (M, \mathbb{R})$. Here $H^1(M, \mathbb{R})$ denotes 
the  Hodge bundle over the moduli space of Abelian differentials  and $H^1_\omega(M,\mathbb{R})$ denotes the fiber at $\omega \in \mathcal H(\alpha)$.

For every $s\in I$, let $\pi_{\omega} (s) : H^1_{h_s\omega} (M, \mathbb{R}) \to H^1_{\omega} (M, \mathbb{R})$ denote the
linear map given by the parallel transport along the horocycle arc joining $h_s\omega$ to $\omega$, that is, 
$\{h_{s-\sigma} \omega\vert \sigma \in [0,s]\}$. Given a horocycle section $\phi: I  \to H^1(M, \mathbb{R})$,
the curve $\phi_\omega: I \to H^1_\omega (M, \mathbb {R})$ is defined as the composition
$$
\phi_\omega (s) =  \pi_\omega(s) \circ \phi (s) \,, \quad \text { for all } s \in I\,.
$$
Since the Kontsevich--Zorich cocycle is defined by parallel transport, the following commutation relation holds:
for all $v \in H^1_{h_s\omega} (M, \mathbb R)$,
$$
\pi_{g_t\omega} (e^{2t}s) \circ  KZ(g_{t},h_s\omega) v =      KZ (g_t, \omega) \circ \pi_\omega(s) v
\,.
$$
It follows that for any horocycle section $\phi:I \to H^1(M, \mathbb R)$ at $\omega$ we have the identity
$$
KZ(g_{t},h_s\omega) \phi(s) = \pi_{g_t\omega}^{-1}  (e^{2t}s) KZ (g_t, \omega)\phi_\omega(s)\,.
$$
In other terms, since $g_t h_s \omega =  h_{e^{2t} s} g_t \omega$  the section
$$
KZ(g_{t},h_{e^{-2t}s}\omega) \phi(e^{-2t}s) = \pi_{g_t\omega}^{-1} (s) KZ (g_t, \omega) \phi_\omega(e^{-2t}s)
$$
is a horocycle section at $g_t\omega$, hence in order to compute the evolution of horocycle sections under the
Teichm\"uller flow it is enough to compute the evolution of curves $\phi_\omega: I \to H^1_\omega(M, \mathbb R)$
under the maps 
$$
\phi_\omega(s) \to    KZ (g_t, \omega) \phi_\omega(e^{-2t}s)\,.
$$
Let $|\cdot|$ denote the Hodge norm.   We recall that the Hodge norm of a cohomology class
$\gamma \in H^1(X, \mathbb C)$  is given by 
$$|\gamma|= \sqrt{ \frac{1}{2} \int_X \gamma \wedge *  \overline {\gamma}} \,,$$ where $*$ denotes the Hodge star operator. This norm depends on the Riemann surface (but not the Abelian differential), though we usually suppress this dependence. See the survey \cite{FMZ} for a detailed description.  For every $g\in SL(2, \mathbb R)$ let $\|KZ(g,\omega)\|$ denote the operator norm with respect to the Hodge norms on $H_\omega^1(M,\mathbb{R})$ and $H_{g\omega}^1(M,\mathbb{R})$. If $\mathcal{F} \subset H^1(M, \mathbb R)$ is a subbundle of the (real) Hodge bundle, let $\|KZ(g,\omega)\|_{\mathcal{F}}$ denote the operator norm with respect to the Hodge norm restricted to $\mathcal{F}$.

\begin{lem} 
\label{lem:KZbound}
Let $\text{dist}$ denote the hyperbolic distance on $SL(2,\mathbb{R})/SO(2)$.   For any Abelian differential $\omega\in \mathcal{H}(\alpha)$ and for all $g\in SL(2,\mathbb{R})$, we have $\|KZ(g ,\omega)\|\leq e^{\text{dist}(g, Id)}$ 
and moreover $|KZ(g,\omega)v|\geq e^{-\text{dist}(g, Id)}|v|$.
\end{lem}
\begin{proof} For $g=g_t$ a diagonal element of $SL(2, \mathbb R)$, the upper bound holds for the Hodge norm by the first variation formulas~\cite[\S 2]{Fo} or \cite[\S 3.5]{FM}. The lower bounds also follows  since $KZ(g_t,\omega)$ is invertible with inverse $K(g_{-t}, g_t\omega)$, hence for $t\geq 0$
$$
|v| = |  KZ(g_{-t}, g_t\omega)     KZ(g_t,\omega)v | \leq e^t   | KZ(g_t,\omega)v | \,.
$$
For a general $g \in SL(2, \mathbb R)$, the result follows from the $KAK$ decomposition of $SL(2, \mathbb R)$,
since the action on the Hodge bundle of the group $K= SO(2) \subset SL(2, \mathbb R)$ of rotations  is isometric. 
\end{proof}

\begin{lem}\label{lem:in compact} For every $c>0$ there is a compact set $\mathcal{K}$ and $t_0>0$ so that if $\omega \in \mathcal{K}$ then 
$|\{-1\leq s \leq 1:g_th_s\omega\in \mathcal{K} \}|>2-c$ for all $t>t_0$. 
\end{lem}
\begin{proof}By \cite[Section 2.3]{ath thesis}, there exists $V$ a continuous, proper function on $\mathcal{H}$, and constants  $c,\gamma, a, b, C$ and $\tau$ so that, for all $t\geq \tau$,
$$\int_0^{2\pi} V(g_tr_\theta\omega)d\theta\leq ce^{-\gamma t}V(\omega)+b$$  and moreover, $a^{-1}V(\omega)<V(g\omega)<aV(\omega)$ for all $g\in SL(2,\mathbb{R})$ with $|g|<C$. 
By choosing $N$ large enough and using that $r_\theta=\hat h_{\tan(\theta)}g_{\log(\cos(\theta))}h_{-\tan(\theta)}$ we have the lemma for $\mathcal{K}=V^{-1}[0,N]$. 
\end{proof}
 
Let $\nu$ be an $SL(2,\mathbb{R})$-invariant measure on a stratum of translation surfaces. We say $\mathcal{F}$ is a $\nu$-\emph{almost everywhere invariant subbundle} if for $\nu$-almost every $\omega$ and any $g \in SL(2,\mathbb{R})$ we have that $g$  sends $\mathcal{F}_\omega$ to $\mathcal{F}_{g\omega}$, that is,
$KZ(g,\omega)\mathcal{F}_\omega = \mathcal{F}_{g\omega}$.

Following \cite[Def. 1.3]{CE} we say that the KZ cocycle has a $\nu$-\emph{measurable almost invariant splitting}
$\mathcal{F}$  if there exists a finite set of proper subbundles $\mathcal{F}_1,\dots,\mathcal{F}_n \subset \mathcal F$ such that   $\mathcal{F}_i \cap \mathcal{F}_j  =\{0\}$ $\nu$-almost everywhere, for all $1\leq i, j \leq n$, and for $\nu$-almost all $\omega$ and almost all $g$, the linear map  $KZ(g,\omega)$ sends  the set 
$\{\mathcal{F}_{1, \omega},\dots, \mathcal{F}_{n, \omega}\}$ to the set $\{\mathcal{F}_{1, g\omega},\dots,\mathcal{F}_{n, g\omega}\}$. 
Following \cite[Def. 1.4]{CE},  we say that the cocycle acts \emph{strongly irreducibly} on $\mathcal{F}$ with respect to the measure $\nu$ if it does not admit any measurable almost invariant splitting.

The span of $ \Re(\omega), \Im(\omega)$  (respectively the real and imaginary part of the Abelian differential  $\omega \in \mathcal H(\alpha)$)  defines a smooth  invariant, symplectic subbundle of the Hodge bundle, which is then $ \nu$-invariant for any $SL(2,\mathbb{R})$ invariant measure $\nu$ on $\mathcal H(\alpha)$.  We call this the \emph{$SL(2,\mathbb{R})$ subbundle.} Let $\hat{\mathcal{F}}$ denote its symplectic complement,  which is also a $\nu$-almost everywhere invariant subbundle.

\subsection{Outline of proof}
\label{subsec:outline}
To prove Theorem \ref{thm:main}, 
most of our work is to rule out the `weak stable space' (in the terminology of \cite{AF}). 
This is different from the approach of  \cite{AD} and our understanding of the approach of \cite{AAD}, where they rule it out for `structural' reasons. It also differs from the approach in \cite{AF}, because it is centered more on moduli space. As in the previous approaches we apply the Veech criterion \cite{veech metric 1} (Lemma \ref{lem:v crit} of this paper) which morally says that it suffices to show that for all 
$\alpha$ there exists $c>0$ and a compact set $\mathcal{K}$ so that for arbitrarily large $t$ such that $g_tr_\theta \omega \in \mathcal{K}$ (we actually need two additional appropriate times in~$\mathcal{K}$)  we have  the following lower bound for the distance from the integer lattice $H^1(M_{g_t\omega}, \mathbb Z))$ 
of the image of the vector $ \alpha \Im(r_\theta\omega)) \in H^1(M, \mathbb R)$ under the Kontsevich--Zorich cocycle:
\begin{equation}\label{eq:star}
\|KZ(g_t,r_\theta \omega) ( \alpha\Im(r_\theta\omega))\|_{\mathbb{Z}}>c\,.
\end{equation}
We show that for every $\omega$ this one parameter family of classes are transverse to any integer translate of the stable $SL(2,\mathbb{R})$-bundle. This allows us to apply Lemma \ref{lem:stay_proj_lip} which makes the assignment $$\theta \to \frac{KZ(g_t,r_\theta \omega)( \alpha \Im(r_\theta\omega))}{|KZ(g_t,r_\theta \omega)( \alpha \Im(r_\theta\omega))|}\,,$$ after rescaling the segment by the geodesic flow which exponentially expands it, closer to a constant curve (for typical $\theta$).  This lets~us apply Proposition \ref{prop:growth mechanism} (a modification of \cite{CE}) to have that   $\|KZ(g_t,r_\theta \omega)\Im(r_\theta\omega)\|$ typically grows in $t$.  The key Proposition \ref{prop:better} shows that, under appropriate assumptions, there exists a $0<\rho<1$ so that for an appropriately chosen segment of angles $J$ we have that for most $\theta \in J$ there is an $s>0$ so that 
$$\|KZ(g_{t+s},r_\theta \omega) ( \alpha \Im(r_\theta\omega))\|_{\mathbb{Z}}>(\|KZ(g_{t},r_\theta \omega) 
( \alpha \Im(r_\theta\omega))\|_{\mathbb{Z}})^\rho$$ and 
$g_{t+s}r_\theta\omega \in \mathcal{K}$.  A key step in the proof (Lemma \ref{lem:key step}) is a new large deviations estimate for the measure  of the set of directions, on any translation surface, where the Kontsevich-Zorich cocycle grows slower than expected. This complements \cite{AAEKMU}, Theorem 1.5, which proves a similar result for when the growth of the cocycle is larger than expected. Iterating Proposition~\ref{prop:better}, we can avoid the issues of `descendants' (contrary to~\cite{AF}, \S 3 especially (3.6), although this terminology is not used there).  Another difference with \cite{AF}'s approach is we treat the ``weak stable''  and stable subbundles at the same time. We also need to treat the unstable-$SL(2,\mathbb{R})$ subbundle and subbundles where the Kontsevich-Zorich cocycle acts isometrically, which we do by straightforward or standard arguments. 

Throughout this proof we treat horocycles instead of circles because they behave better under the geodesic flow, and relate horocycles to circles in Section~\ref{sec:stab and iso}. Section~\ref{sec:billiard} proves 
Theorem~\ref{thm:billiard} via the previously mentioned, and now standard, Baire Category argument. In this section we follow the approach of \cite{vorobets}.

\section{Making vectors grow}\label{sec:grow}
This section proves Propositions~\ref{prop:growth mechanism} and~\ref{prop:easy restriction} 
 and  then develops, in a straightforward way, some machinery from \cite{EskMirMo}. 

Let $\FF$ denote the complementary subbundle,  that is, the symplectic orthogonal, of the $SL(2,\mathbb{R})$  subbundle and let $\nu$ be any $SL(2, \mathbb R)$-invariant ergodic probability measure.

\begin{prop}\label{prop:growth mechanism}(Chaika-Eskin, \cite{CE})  If $\mathcal{F}\subset \FF$ is an equivariant subbundle, where the Kontsevich-Zorich cocycle acts strongly irreducibly and with a positive exponent, then there exists 
$\lambda:= \lambda_\mathcal F>0$ such that,
for all $\delta,\epsilon >0$,   and for all $L$ sufficiently large,  there exist  $\hat\epsilon>0$ and an open set $\hat{U}:= \hat U(\delta, L)$ so that 
$\nu(\hat{U})>1-\delta$, and for all $\omega\in \hat{U}$ and for all parallel horocycle sections $v$ at 
$\omega$, that is, for the all the maps $v : [-1,1] \to \mathcal{F}$ such that $v(s) \in 
\mathcal{F}_{h_s\omega}$ and $\pi_\omega (s)v(s) =  v(0)$,   we have 
$$
\begin{aligned}
\text{\rm Leb}(\{-1 \leq s\leq 1:&|KZ(g_L,h_s\omega)u)| \geq e^{ {\lambda} L (1-\epsilon)}  \vert u\vert\,, \text{ for all } u \in \mathcal F_{h_s\omega} \\ & \text{ with } \angle(u,v(s))<\hat\epsilon\})>2 -\delta. 
\end{aligned} $$  
\end{prop}

Proposition~\ref{prop:growth mechanism} is proved below.

For any compact subset $\mathcal K$, for any $\omega$ and $t>0$,  let 
$$
v_{\mathcal K} (\omega, t) = \vert \{ \tau \in [0, t] \vert  g_\tau (\omega) \in \mathcal K\} \vert\,.
$$

\begin{prop}\label{prop:easy restriction}
Let $\mathcal{F}\subset \FF$ be an equivariant subbundle of  $\FF$. For any compact subset $\mathcal K$
there exists $\lambda:=\lambda_{\mathcal{F}, \mathcal K} \in [0,1)$  such that 
\begin{itemize}
\item $\sup_{v\in {\mathcal{F}_\omega}}\frac{|KZ(g_{t},\omega)v|}{ |v| } \leq  \exp \left( t - (1-\lambda )v_{\mathcal K} (\omega, t) \right)$;
\item $\inf_{v\in {\mathcal{F}_\omega}}  \frac{|KZ(g_{t},\omega)v|}{ |v| } \geq  \exp \left( -t + (1-\lambda )v_{\mathcal K} (\omega, t) \right)$.
\end{itemize}
\end{prop}

\begin{proof}  Both estimates follow from the first variational formulas for the Hodge norm~\cite[\S 2]{Fo}
or  \cite[\S 3.5]{FM}.
\end{proof}

We now prove Proposition~\ref{prop:growth mechanism}. 
\subsection{Proof of Proposition \ref{prop:growth mechanism}}
Let $\lambda>0$ denote the largest element of the Lyapunov spectrum restricted to $\mathcal{F}$. 

Let $\mu$ be an $SO(2)$ invariant, compactly supported probability measure on $SL(2,\mathbb{R})$, whose support generates a dense subgroup of $SL(2,\mathbb{R})$. 
Let $\mathcal M$ be any $SL(2, \mathbb R)$-invariant suborbifold. We recall that  Eskin and M. Mirzakhani \cite{EskMir} and A. Eskin, M. Mirzakhani and A. Mohammadi \cite{EskMirMo} proved that all 
orbit closures of the $SL(2,\mathbb{R})$ action on the moduli space are ``affine"  suborbifolds supporting a
unique ``affine'' probability invariant measure. 

For any $L\in \mathbb N \setminus\{0\}$ and $\epsilon>0$, let $E_{good}(L,\epsilon)$ be the set of all surfaces $\omega \in \mathcal M$ so that for every 
$ v \in \mathcal{F}_{\omega}$ there exists a subset $H(v) \subset SL(2, \mathbb{R})^L$
$$\mu^L
(H(v)) > 1 - \epsilon,$$
and for all $(h_1, \dots , h_L) \in H(v)$, we have 
$$e^{(\lambda-\epsilon )L} |v|
<
|KZ(h_L\cdots h_1,\omega)v|
< e^{(\lambda+\epsilon)L} |v|$$

\begin{lem}\label{lem:Egood}[Chaika-Eskin \cite{CE}, Lemma 2.11]   For all $\epsilon>0$ we have $\underset{L \to \infty}{\lim}\nu(E_{good}(L,\epsilon))=1$. 
\end{lem}
Note that by assumption the cocycle acts strongly irreducibly on $\mathcal{F}$ justifying the application of \cite{CE}. 
The continuity of the Hodge norm implies that $E_{good}(L,\epsilon)$ is open.

To prove the proposition we need to relate random walks to Teichm\"uller geodesics via standard techniques:

\begin{lem}[Sublinear Tracking]
\label{lemma:sublinear:tracking}
There exists $\lambda >0$ (depending only on $\mu$), and for
$\mu^{\mathbb{N}}$-almost all $\bar{g} = (g_1, \dots, g_n, \dots) \in SL(2,\mathbb{R})^{\mathbb{N}}$ there
exists $\bar{\theta} = \bar{\theta}(\bar{g}) \in \mathbb{R}$ such that 
\begin{equation}
\label{eq:lemma:sublinear:tracking}
\lim_{n \to \infty} \frac{1}{n} \log \| (g_{\lambda n} r_{\bar{\theta}})
(g_n \dots g_1)^{-1} \| = 0.
\end{equation}
 Furthermore, the
distribution of $\bar{\theta}$ is uniform, i.e
\begin{equation}
\label{eq:uniform:hitting:measure}
\mu^{\mathbb{N}}\left( \{ \bar{g} \in SL(2,\mathbb{R})^{\mathbb{N}} :
  \bar{\theta}(\bar{g}) \in [\theta_1, \theta_2] \} \right) =
\frac{|\theta_2 - \theta_1|}{2\pi}.
\end{equation}
\end{lem}
Note that \eqref{eq:uniform:hitting:measure} follows from the fact that $\mu$ is $SO(2)$ invariant. 

Because, when $\theta \notin \{\pm \frac \pi 2\}$,
\begin{equation}\label{eq:NAN}
r_\theta=\begin{pmatrix} 1&0\\ \tan(\theta)&1\end{pmatrix}\begin{pmatrix} \cos(\theta)&0\\0&\sec(\theta)\end{pmatrix}\begin{pmatrix} 1& -\tan(\theta)\\0&1\end{pmatrix}
\end{equation}
the previous lemma implies:

\begin{lem}[Sublinear Tracking]
\label{lem:sublin hor}
There exists $\lambda >0$ (depending only on $\mu$), and 
$\mu^{\mathbb{N}}$-almost all $\bar{g} = (g_1, \dots, g_n, \dots) \in SL(2,\mathbb{R})^{\mathbb{N}}$ there
exists $\bar{s} = \bar{s}(\bar{g}) \in(-\infty, + \infty)$ such that 
\begin{equation}
\label{eq:lemma:sublin hor}
\lim_{n \to \infty} \frac{1}{n} \log \| (g_{\lambda n} h_{\bar{s}})
(g_n \dots g_1)^{-1} \| = 0.
\end{equation}
 Furthermore, the
distribution of $\bar{s}$ is in the measure class of Lebesgue.
\end{lem}

Now let $\hat{E}_{good}(L,\tilde{\epsilon})$ 
 be the set of all surfaces $\omega\in \mathcal M$ so that, for every 
$ v \in \mathcal F_{\omega}$; we have  
$$\text{\rm Leb}(\{s\in [-1,1]:
e^{(\lambda-\tilde{\epsilon} )L}
<
\frac{ |KZ(g_L,h_s\omega)v_{\omega'}(s)|}{ | v_{\omega'}(s) |}
< e^{(\lambda+\tilde{\epsilon})L}\})>  2-\delta\,.$$

Lemmas \ref{lem:sublin hor} and  \ref{lem:Egood} imply:
\begin{lem}\label{lem:good growth} $\underset{L \to \infty}{\lim} \nu(\hat{E}_{good}(L,\tilde\epsilon))=1.$
\end{lem}
We conclude the proof of Proposition \ref{prop:growth mechanism}.  Let $\delta >0$ and $0<\tilde{\epsilon}< \frac 1 2 \epsilon$. 
By Lemma~\ref{lem:good growth}, for $L$ large enough we have that 
$\nu(\hat{E}_{good}(L,\tilde\epsilon)) >  1 -\delta$ and  $e^{\frac 1 2 \epsilon L}>8$. 

 We then choose $\hat{U}:=\hat{E}_{good}(L,\tilde\epsilon)$ which is an open set by continuity of the Hodge norm.  

By construction, for all $\omega \in U$ and for all parallel horocycle sections $v$ at 
$\omega$, that is for the all the maps $v : [-1,1] \to \mathcal{F}$ such that $v(s) \in 
\mathcal{F}_{h_s\omega}$ and $\pi_\omega (s)v(s) =  v(0)$, we have
$$\text{\rm Leb}(\{s\in [-1,1]: e^{(\lambda-\tilde{\epsilon} )L}
< \frac{ |KZ(g_L,h_s\omega)v(s)|}{ | v(s) |}
< e^{(\lambda+\tilde{\epsilon})L}\})>  2  -\delta\,.$$
Let now $u \in V_{h_s\omega}$ be a vector such that $\angle (u, v(s)) < \hat \epsilon$. We can write $u = a v(s) +  b v^\perp(s)$ with $v^\perp(s) \perp v(s)$ (with respect to the Hodge inner product) and by assumption   we have $$\vert a \vert \geq \frac{ 1-\hat \epsilon}{2} \frac{| u|}{|v(s)|}
\quad \text{ and } \quad \vert b \vert \leq 2 \hat \epsilon \frac{| u|}{ | v^\perp(s)|} \,.$$
It follows that,   {if $\hat \epsilon < e^{-L} e^{(\lambda - \tilde \epsilon) L} / 8$, by applying Lemma~\ref{lem:KZbound},  we have 
$$
|KZ(g_L,h_s\omega)u| \geq  e^{(\lambda-\tilde{\epsilon} )L}  \frac{ 1-\hat \epsilon}{2} | u|   -
2 \hat \epsilon   e^{L}   |u| \geq  \frac{e^{(\lambda-\tilde{\epsilon} )L}}{8}|u| > e^{(\lambda-\epsilon)L}|u|\,. 
$$
The argument is completed.

\subsection{Developing} \label{sec:EMM stuff}

\begin{thm}(Eskin-Mirzakhani-Mohammadi)\label{thm:EMM 2.7} Let $U$ be an open set, $\nu$ be an $SL(2,\mathbb{R})$ invariant and ergodic measure. 
For any $\epsilon>0$ there exists a finite set of invariant manifolds, $\mathcal{Z}_1,\dots,\mathcal{Z}_n$ so that for any  
$\mathcal{C} \subset \supp(\nu) \setminus \cup_{i=1}^n\mathcal{Z}_i$, compact, there exists $T_1>0$ so that for all $\omega \in \mathcal{C}$  and  for all $T\geq T_1$, 
$$\frac 1 {2T} \int_{-1}^1\int_0^T\chi_U(g_th_s\omega)dtds>\nu(U)-\epsilon.$$
\end{thm}
This follows from \cite[Theorem 2.7]{EskMirMo} by choosing $\phi\in C_c(\mathcal{H})$, 
with $supp(\phi)\subset U$, $0\leq \phi\leq 1$ with $\|\phi\|_1>\nu(U)-\frac \epsilon2$. 

\begin{thm}(\cite[Proposition 2.13]{EskMirMo}) \label{thm:EMM 2.13} Let $\mathcal{M}\subset \mathcal{H}$ be an affine invariant submanifold. (In this proposition $\mathcal{M}=\emptyset$ is allowed.) Then there exists an $SO(2)$-invariant function 
$f_{\mathcal{M}}: \mathcal{H} \to [1,\infty]$  with the following properties:
 \begin{enumerate}
 \item $f_{\mathcal{M}}(\omega)=\infty$ if and only if $\omega\in \mathcal{M}$, and $f_{\mathcal{M}}$ is bounded on compact subsets of $\mathcal{H}\setminus \mathcal{M}$. For any $\rho>0$, the set 
 $\{\omega : f_{\mathcal{M}}(\omega)\leq \rho\}$ is a compact subset of 
 $\mathcal{H}\setminus \mathcal{M}$.
\item There exists $b > 0$ (depending on $\mathcal{M}$) and for every $0 < c < 1$ there exists $t_0 >0$ (depending on $\mathcal{M}$ and $c$) such that for all $\omega\in \mathcal{H}\setminus \mathcal{M}$ 
and all $t>t_0$, 
$$\frac 1 {2\pi} \int_0^{2\pi} f_{\mathcal{M}}(g_tr_{\theta}\omega)d\theta \leq cf_{\mathcal{M}}(\omega) + b.$$
\item  There exists $\sigma > 1$  and $V\subset SL(2,\mathbb{R})$ a neighborhood of the identity so that 
for all $g\in V$ and all $\omega \in \mathcal{H}$,
$$\sigma^{-1}f_{\mathcal{M}}(\omega) \leq  f_{\mathcal{M}}(g\omega) \leq \sigma f_{\mathcal{M}}(\omega).$$
\end{enumerate}
\end{thm}

This implies a similar result for horocycles:

\begin{lem}\label{lem:2.13 for hor}\cite[Lemma 3.5]{AAEKMU} Let $f_{\mathcal{M}}$ be as in Theorem \ref{thm:EMM 2.13}. Then there exists a constant $b'>0$ so that for all $0<a<1$ there exists $\bar{t}_0= \bar{t}_0(a)$ such that for all $t> \bar{t}_0$ and for all $\omega \in \mathcal{H}\setminus \mathcal{M}$ we have 
$$\int_{-1}^1 f_{\mathcal{M}}(g_th_s \omega)ds <  a f_{\mathcal{M}}(\omega)+  b'\,.$$
\end{lem}

For each $\rho \in \mathbb{R}^+$, let $\mathfrak{C}_\rho=\{\omega:f_{\mathcal{M}}(\omega)\leq \rho\}$ and, for
all $N\in \mathbb N$, let
$$
Z_N=\Big\{s\in[-1,1]: g_{j t}h_s\omega \notin \mathfrak{C}_\rho \text{ for all }j \in \, \{1,2,\dots,N\}\Big\}\,.
$$

\begin{prop} \label{prop:cpct integer} \cite[Proposition 3.7]{AAEKMU} 
Let $f_{\mathcal{M}}$ be as in Theorem \ref{thm:EMM 2.13} and let $b'>0$ and $\bar{t}_0=\bar{t}_0(a)$ be as in Lemma \ref{lem:2.13 for hor}. There exist $C_1> 1$ (independent of
$\omega$ and $a$) such that for all $a\in (0,1)$, all $\rho > C_1 b'/a$,  all $t\geq \bar{t}_0$ such that $e^t\in \mathbb N$ and all $N\in \mathbb N$, 
$$\int_{Z_{N-1}}  f_{\mathcal{M}}  (g_{Nt}h_s\omega)ds<  (2a)^{N} f_{\mathcal{M}}(\omega)+(2a)^{N-1}b'\,.$$
\end{prop}

By choosing $a<\min \{ \frac 1 2, \,  \frac 1 {2b'}\}$ we obtain:
\begin{cor} \label{cor:ath} (cf \cite[Theorem 1.1 (2)]{ath thesis}) There exists $\rho\in \mathbb{R}, \zeta<1$ so that for all $S$, $T$ large enough if $\omega \in \mathfrak{C}_\rho$
$$\text{\rm Leb}(\{s\in [-1,1] : \cup_{t\in [S,S+T]}\{g_t  h_s  \omega\} \cap  \mathfrak{C}_\rho=\emptyset\})<\zeta^T.$$
\end{cor}

\begin{prop}\label{prop:in good integer}\cite[Proposition 3.9]{AAEKMU} (cf  \cite[Theorem 1.5]{KKLM}) Let $\mathfrak{C}_\rho$ be as above for a function $f_{\mathcal M}$ as in Theorem \ref{thm:EMM 2.13}.  
Let us assume that $\omega \not \in \mathcal M$. For any $\delta,a \in (0,1)$ there exist $\rho_0>1$ and $t_0 >1$, depending only on $a$, such that for all  $\rho \geq  \rho_0$, all $t \geq t_0$ such that $e^t\in \mathbb N$,  and
 all $N\in \mathbb{N}$, the set
$$\{s\in [-1,1]: \frac 1 N\sum_{j=1}^{N}\chi_{\mathfrak{C}_\rho}(g_{jt}h_s\omega)<1- \delta\}$$
can be covered by at most $2^N C_1^N (2a)^{\delta N} e^{2tN} C(\omega)$ intervals of radius $e^{-2tN}$, where
$C(\omega)=\max\{ 1,f_{\mathcal M}(\omega)/\rho  \}$.
\end{prop}
Note that, if  $a<\frac 1 {2 (2C_1)^{\frac 1 \delta}}$, the measure of the set in the statement of Prop. \ref{prop:in good integer} decays exponentially in $N$.

\begin{cor} \label{cor:all together} Let $U$ be an open set whose boundary has measure zero.  For all $\epsilon>0$ there exist numbers $\xi\in (0,1)$, $t_0>0$, $N_0\in \mathbb{N}$ and a set  $V_\epsilon := V_\epsilon(U) \subset \supp (\nu)$, open in 
$\supp (\nu)$,  and $T_2\in \mathbb{R}^+$  so that for any $T>T_2$:
\begin{enumerate}[label=(\roman*)]
\item $\nu(V_\epsilon)>1-\epsilon;$
\item for $\nu$ almost every $\omega$ we have $\underset{S \to \infty}{\liminf}\,  \frac{|\{ 0\leq t\leq S: g_t\omega\in V_\epsilon\}|}{S}>1-\epsilon;$
\item if $\omega \in V_\epsilon$   we have $|\frac 1 {2T} \int_{-1}^1\int_0^T \chi_U(g_th_s\omega)dtds-\nu(U)|<\epsilon;$
\item\label{c:large dev} for all $t\geq t_0$, $N\geq N_0$ and $\omega  \in V_{\epsilon}$ we have 
$$\big|\{s \in [-1,1]:\frac 1 N\sum_{j=1}^{N}\chi_{V_\epsilon}(g_{jt}h_s\omega)<1- \epsilon \}\big|<\xi^N\,.$$
\end{enumerate}
\end{cor}
This uses the following standard definition and result:
\begin{defin} Let $\omega \in \mathcal{H}$ and $\nu$ be the unique, $SL(2,\mathbb{R})$-invariant and ergodic probability measure with $supp(\nu)=\overline{SL(2,\mathbb{R})\omega}$. We say $\omega$ is \emph{Birkhoff generic} if 
$$\underset{T \to \infty}{\lim}\, \frac 1 T\int_0^T\phi(g_t\omega) dt=\int_{\mathcal{H}}\phi d\nu$$
for all $\phi \in C_c(\mathcal{H})$.
\end{defin}
\begin{lem}\label{lem:gen set}
Let $S$ be a set so that $\nu(\partial S)=0$. If $\omega$ is Birkhoff generic then 
$$\underset{T \to \infty}{\lim}\, \frac 1 T\int_0^T\chi_S(g_t\omega) dt=\nu(S).$$
If $U$ is an open set and $\omega$ is Birkhoff generic
$$\underset{T \to \infty}{\liminf} \, \frac 1 T\int_0^T\chi_U(g_t\omega) dt\geq \nu(U).$$
\end{lem}

\begin{proof} [Proof of Corollary \ref{cor:all together}]  First, (ii)  follows free of charge by the fact that $V_\epsilon$ is open and Lemma \ref{lem:gen set}.
 Given an open set $U$ and $\epsilon>0$, we then construct an open set $V_\epsilon$ with properties (i), (iii) and (iv). 
 
 For the given $U$ and $\epsilon>0$, we apply Theorem \ref{thm:EMM 2.7} to obtain a finite number of manifolds $\mathcal{Z}_1,\dots,\mathcal{Z}_n$. 
For each of these manifolds we build a function $f_i$ as in Theorem \ref{thm:EMM 2.13} and let $\mathfrak{C}^{(i)}_s=\{\omega:f_i(\omega) \leq s\}$.

Applying Proposition \ref{prop:in good integer} to each $f_i$ with $\delta<\frac \epsilon n$ and $a<\frac 1 {2 (2C_1)^{\frac 1 \delta}}$ (where $C_1$ is as in the proposition),  there exists $\rho_i$, $t_i$ 
as in the proposition, and so that $\nu(\cap \mathfrak{C}_{\rho_i}^{(i)})>1-  \epsilon $. Following the sentence after the proposition,   for any $\sigma>1$ there exist constants $\zeta_1,\dots,\zeta_n<1$ and natural numbers $N_1,\dots,N_n$ so that for each $i\in \{1,\dots,n\}$,  for all  $t\geq t_i$,  for all  $\omega \in \mathfrak{C}_{\sigma\rho_i}^{(i)}$ and $N\geq N_i$,
$$\big|\{s \in [-1,1]:\frac 1 N\sum_{j=1}^{N}\chi_{\mathfrak{C}_{\rho_i}}(g_{jt}h_s\omega)<1-  \frac \epsilon n   \}\big|<\zeta_i^N .$$ Now, letting $\xi>\max\{\zeta_i\}$, there exist $N_0, t_0$ so that,  if $\omega \in  \mathfrak{C}_{\sigma\rho_i}^{(i)}$ for all $i \in \{1, \dots, n\}$, then for all $N\geq N_0$ and $t\geq t_0$  we have 
$$\big|\{s \in [-1,1]:\frac 1 N\sum_{j=1}^{N}\chi_{\cap \mathfrak{C}_{\rho_i}}(g_{jt}h_s\omega)<1- \epsilon \}\big|<\xi^N\,.$$
 We define the set $V_\epsilon$ to be the interior in $\supp(\nu)$  of the compact set $\mathcal C=\supp(\nu) \cap  \cap_{i=1}^n \mathfrak{C}_{\sigma\rho_i}^{(i)}$, so
that property (iv) holds by the above argument.
By applying Theorem \ref{thm:EMM 2.7} to the compact set $$\mathcal C := \supp(\nu)  \cap  \bigcap_{i=1}^n\mathfrak{C}_{\sigma\rho_i}^{(i)}$$ 
and the given $\epsilon>0$, we conclude that there exists $T_2>0$ such that property (iii) holds for all $T>T_2$.  Finally,  it follows from the definitions that $\nu (V_\epsilon) \geq \nu(\cap \mathfrak{C}_{\rho_i}^{(i)})>1- \epsilon $, hence property (i) is also proved.
\end{proof}
We will of course be interested in $V_\epsilon$ chosen for $U=\hat{U}$ and this is what $\hat V_\epsilon$ denotes for the remainder of the paper.
We further assume that $\hat V_\epsilon$ is contained in a fixed compact set.  

\section{ Some preparatory results}\label{sec:lip and stuff}
\subsection{Making curves grow}
\label{subsec:curve_grow}

The main result of the next section is Lemma \ref{lem:Zgrowth from V}, which applies Proposition \ref{prop:growth mechanism} to say that the geodesic flow image of  (projectively) Lipschitz horocycle sections with small enough (projective) Lipschitz constants in a subbundle as in Proposition \ref{prop:growth mechanism} typically grow. It also collects some results which say that the geodesic flow image of horocycle sections have their Lipschitz (and usually their projective Lipschitz) constants improve.

\begin{defin}
A  horocycle section $\phi: [a,b] \to H^1(M,\mathbb{R})$ at $\omega$ (that is, a map $\phi: [a,b] \to H^1(M,\mathbb{R})$ such that $\phi(s) \in H_{h_s\omega}(M, \mathbb R)$ for all $s\in [a,b]$)  is 
$K$-Lipschitz
at $\omega$ if the curve $\phi_{\omega}: [a,b] \to H^1_{\omega}(M, \mathbb R)$, obtained by parallel transport along the horocycle,  (that is, such that $\phi_\omega(s)=\pi_\omega(s) \circ \phi(s)$ for all $s\in [a,b]$) is Lipschitz with respect to the Hodge norm on $H^1_\omega(M, \mathbb R)$ with Lipschitz constant $K>0$.
\end{defin}

\begin{lem}\label{lem:stay lip}
If $\phi:[a,b] \to H^1(M,\mathbb R)$ is a $K$-Lipschitz horocycle section at $\omega$, then $KZ(g_t)(\phi):[e^{2t}a,e^{2t}b] \to H^1(M, \mathbb R)$ given by $$KZ(g_t)(\phi)(s)=KZ(g_t,h_{se^{-2t}}\omega) \phi(se^{-2t})$$ is a $Ke^{-t}$-Lipschitz horocycle section
at $g_t\omega$.
\end{lem}
\begin{proof}
By definition $\phi:[a,b] \to H^1(M,\mathbb R)$ is a $K$-Lipschitz horocycle section at $\omega$ if the parallel transport $\phi_\omega: I \to H^1_\omega(M,\mathbb R)$ is a $K$-Lipschitz map with respect to the Hodge norm. We recall that the map 
$\phi_\omega$
is defined, for all $s\in I$,  as a composition $\phi_\omega (s) = \pi_\omega(s) \circ \phi(s)$ with the parallel transport $\pi_\omega (s): H^1_{h_s\omega}(M,\mathbb R) \to H^1_\omega(M,\mathbb R)$. 

By the commutativity of parallel transport and the KZ cocycle, we have
$$
\pi_{g_t\omega} (s) KZ(g_t,h_{se^{-2t}}\omega)\phi(se^{-2t}) = KZ(g_t,\omega)\phi_\omega (se^{-2t})\,,
$$
hence by Lemma \ref{lem:KZbound}  the map defined as $KZ(g_t,\phi)_{g_t\omega}(s) = KZ(g_t,\omega)
\phi_\omega (se^{-2t})$ is 
$Ke^{-t}$-Lipschitz.

\end{proof} 

\begin{defin} A horocycle section $\phi: [a,b] \to H^1(M,\mathbb{R})$ at $\omega$ is projectively 
$\kappa$-Lipschitz at $\omega$ if there exists $K >0$ such that the curve $\phi_{\omega} : [a,b] \to H^1_\omega(M, \mathbb R)$ is  $K$-Lipschitz at $\omega$ and the Hodge norm $\Vert \phi_\omega(s) \Vert_\omega$ is bounded below by $K/\kappa >0$, for all $s\in [a,b]$.  \end{defin}

\begin{rem} \label{rem:Lipschitz_1}  Using Lemma \ref{lem:KZbound} to bound the change of Hodge norm  from above and below under parallel transport along a horocycle, we obtain that there exists $C_{a,b}$, depending only on $b-a$ so that for any section $\phi: [a,b] \to H^1(M,\mathbb{R})$ projectively $\kappa$-Lipschitz  at $\omega$,  and for any $s \in [a,b]$, the  section $\phi$ is also 
projectively $C_{a,b} \kappa$-Lipschitz at $h_{s}\omega$.
\end{rem} 

\begin{rem}  \label{rem:Lipschitz_2}  For any $K>0$, every $K$-Lipschitz horocycle section $\phi: [a,b] \to H^1(M,\mathbb{R})$ at $\omega$, is projectively $\kappa$-Lipschitz at $\omega$ if $ \max_{s\in [a,b]}  \Vert \phi_\omega (s) \Vert_\omega  
> K (b-a)$ with 
$$
\kappa :=  K / \big(\max_{s\in [a,b]}  \Vert \phi_\omega (s) \Vert_\omega - K(b-a) \big)\,.
$$
\end{rem} 
\begin{lem}\label{lem:stay_proj_lip}
If $\phi:[a,b] \to H^1(M,\mathbb R)$ is a projectively $\kappa$-Lipschitz horocycle section at $\omega$, then 
$KZ(g_t)(\phi):[e^{2t}a,e^{2t}b] \to H^1(M, \mathbb R)$ given by $KZ(g_t)(\phi)(s)=KZ(g_t,h_{se^{-2t}}\omega) \phi(se^{-2t})$  is a projectively $\kappa$-Lipschitz horocycle section at $g_t\omega$. If in addition $\phi:[a,b] \to 
\mathcal F \subset \FF$,  then  for any compact set $\mathcal K$ there exists $\lambda_{\mathcal F, \mathcal K}\in [0,1)$ such that the following holds. For any $\omega$ and for any $t>0$, let 
$$
v_{\mathcal K} (\omega, t) := \vert \{ t' \in [0,t] :   g_{t'} \omega \in \mathcal K\}\vert.
$$
Then for all $t>0$ the horocycle section $KZ(g_t)(\phi): [e^{2t}a, e^{2t}b] \to \FF$ is  projectively $\kappa_t$-Lipschitz at $g_t\omega$, where 
$$
\kappa_t \leq    \kappa \exp (-(1-\lambda_{\mathcal F,\mathcal K})  v_{\mathcal K} (\omega, t))\,. 
$$
\end{lem}
\begin{proof} By Lemma~\ref{lem:stay lip}, if the horocycle section $\phi$ is  $K$-Lipschitz at $\omega$, 
then $KZ(g_t)(\phi)_{g_t\omega}$ is $Ke^{-t}$-Lipschitz, and by Lemma~\ref{lem:KZbound} it follows that  $\Vert KZ(g_t)(\phi)_{g_t\omega} (s) \Vert \geq e^{-t} 
\Vert  \pi_\omega(\phi) (s) \Vert$, for all $s\in [a,b]$. It follows immediately that  if 
$\Vert \pi_\omega(\phi) (s) \Vert_\omega \geq K/\kappa$, then
$$
\Vert KZ(g_t)(\phi)_{g_t\omega} (s) \Vert \geq e^{-t} 
\Vert  \phi_{\omega} (s) \Vert \geq  K e^{-t}/\kappa\,,
$$
hence $KZ(g_t)(\phi)$ is projectively $\kappa$-Lipschitz at $g_t\omega$ establishing the first claim. 

\noindent If $\phi:[a,b] \to \FF$, it follows from Proposition \ref{prop:easy restriction} that there exists
$\lambda:=\lambda_{\mathcal F, \mathcal K} \in [0,1)$ such that
\begin{multline}
\Vert KZ(g_t)(\phi)_{g_t\omega} (s)\Vert = \Vert KZ(g_t, \omega)(\phi_{\omega}(s))\Vert  \\  \geq   
e^{- t + (1-\lambda) v_{\mathcal K} (\omega,t)  }
 \Vert \phi_{\omega}(s)\Vert \,.
\end{multline}
Since $\phi$ is projectively  $\kappa$-Lipschitz at $\omega$, 
there exists $K>0$ such that $\phi$ is $K$-Lipschitz at $\omega$ and satisfies the lower bound 
$\Vert \phi_{\omega}(s) \Vert  \geq  K /\kappa$, for all $s\in [a,b]$.
It follows that
\begin{multline}
\Vert KZ(g_t)(\phi)_{g_t \omega}  (s)\Vert  \geq    (K / \kappa) e^{- t + (1-\lambda) v_{\mathcal K} (\omega,t)  }  \\  = K e^{-t} / (\kappa e^{- (1-\lambda) v_{\mathcal K} (\omega,t)  })\,,
\end{multline}
which implies that  the section $KZ(g_t)(\phi)$ is projectively  $\kappa_t$-Lipschitz  at $g_t \omega$ with 
$\kappa_t \leq   \kappa e^{ -(1-\lambda) v_{\mathcal K} (\omega,t)  } $, as stated.

\end{proof}

\begin{lem}\label{lem:growth from E} Let $\mathcal{F} \subset \hat{\mathcal{F}}$ be a subbundle where $KZ$ acts strongly irreducibly and with a positive exponent. There exists $\lambda >0$ such that the following holds. For any $\delta, \epsilon>0$ and for all sufficiently large $L>0$, there exist $\kappa_0:=
\kappa_0 (a,b,\delta, L)$ and an open set $\hat U:=\hat U(\delta, L)$  (as in Proposition \ref{prop:growth mechanism})  with $\nu(\hat U)>1-\delta$,  such that for all $\kappa \in (0, \kappa_0)$,  for all  $t\geq 0$, for all $s_0 \in [a+e^{-2t},b-e^{-2t}]$ and for all horocycle sections 
$\phi:[a,b]\to \mathcal{F}\subset \FF$  projectively $\kappa$-Lipschitz  at $h_{s_0} \omega$, whenever  $g_th_{s_0}\omega\in \hat{U}$,   we have  that 
\begin{multline}
\text{\rm Leb}(\{s\in [s_0- e^{-2t},s_0+e^{-2t}]:  \\ |KZ(g_{t+L},h_s\omega)\phi(s)|< e^{ {\lambda} L(1-\epsilon)}|KZ(g_t,h_s\omega)\phi(s)|\})<\delta e^{-2t}.
\end{multline}
\end{lem}

\begin{proof}
The lemma follows by Proposition~\ref{prop:growth mechanism} applied to the horocycle 
$\{ h_s g_t (h_{s_0} \omega)  \vert  s \in [-1, 1]\}$  at $g_t h_{s_0}\omega$. Let us consider the  image  
$KZ(g_t) (\phi)$ of the section $\phi$ for $s\in [s_0-e^{-2t}, s_0 + e^{-2t}]$.  By definition this section is the 
restriction of a horocycle section at $g_t\omega$ and after reparametrization can be regarded as a horocycle
section at $g_t h_{s_0} \omega= h_{e^{2t}s_0} g_t\omega$, defined as $KZ(g_t, h_{s_0+e^{-2t}s}\omega) \phi(s_0+ e^{-2t}s)$ for $s\in  [-1,  1]$.  
 By Lemma~\ref{lem:stay_proj_lip}  since $\phi$ is projectively
$\kappa$-Lipschitz at $h_{s_0} \omega$, then $KZ(g_t) (\phi)$, and its reparametrization,  are still projectively 
$\kappa$-Lipschitz at $g_th_{s_0}\omega$. It follows that the  ratio $K/m$ between the
Lipschitz constant $K$ and the minimum Hodge norm $m$ of the curve   
$$
KZ(g_t)(\phi)_{g_t h_{s_0}\omega}(s) = \pi_{g_t h_{s_0}\omega} ( KZ(g_t)(\phi)(s)   ) \,, \quad \text{for }s\in  [-1,  1]\,,
$$
 is at  most $\kappa$, hence, for all $s\in [-1,1]$ we have that
$$
\angle\Big(KZ(g_t)(\phi)_{g_th_{s_0}\omega} (s), KZ(g_t)(\phi)_{g_th_{s_0}\omega} \phi (s_0)\Big)<  \kappa
$$
hence there exists $\kappa_0>0$ such that, for $\kappa \in (0, \kappa_0)$, the horocycle section 
$$KZ(g_t)(\phi) (e^{2t}s_0+s)=KZ(g_t, h_{s_0+ e^{-2t} s}\,\omega) \phi(s_0+e^{-2t}s)$$ 
 makes an angle at most $\hat \epsilon$ (with respect to the Hodge norm on $H^1_{h_sg_t h_{s_0} \omega}(M,\mathbb R)$) 
 with a parallel section, so that  since
$$
KZ(g_{t+L}, h_{s_0+e^{-2t}s}\omega)\phi (s_0+e^{-2t} s) = KZ(g_L,h_s  g_t h_{s_0}\omega) KZ(g_t)(\phi) (e^{2t}s_0+  s)
$$
by Proposition~ \ref{prop:growth mechanism} we have 
$$
\begin{aligned}
\text{\rm Leb}(\{-1 \leq s\leq 1:&|KZ(g_{t+L},h_{s_0+ e^{-2t} s}\,\omega) \phi (s_0+e^{-2t} s)| \\  &\geq e^{ \lambda L(1-\epsilon)} | KZ(g_t, h_{s_0+ e^{-2t} s}\,\omega) \phi(s_0+e^{-2t}s)|  \})> 2-\delta. 
\end{aligned} $$ 
which  implies the statement by change of variables.

\end{proof}
The main result of this section is the following lemma.
Let $\delta >0$ and let $L>0$ be sufficiently large. 

Let $\hat U:= \hat U(\delta, L)$ be the open set given by Proposition~\ref{prop:growth mechanism} and Lemma~\ref{lem:growth from E},  
and,  for all small $\epsilon>0$, let $\hat V_\epsilon:=V_\epsilon 
(\hat U(\delta,L)) \subset \supp(\nu)$ denote the open subset (which depends on $\hat U(\delta,L)$) given by Corollary \ref{cor:all together}).

\smallskip
\noindent
\textbf{Notation:}  For all $\epsilon, \delta>0$, let
\begin{equation}
\begin{aligned}
  \label{eq:eta/kappa}
 \eta (\epsilon, \delta)&:=\sqrt{3(\epsilon+\delta)}+2\sqrt{\delta},  \quad \text{ and }  \\
  \mu (\epsilon, \delta)&:= {\lambda}(1-\epsilon)   - \big( 1+  {\lambda}(1-\epsilon) \big) \eta(\epsilon, \delta)  \,.
\end{aligned}
\end{equation}

\begin{lem} \label{lem:Zgrowth from V}  
Let $\mathcal{F}\subset \hat{\mathcal{F}}$ be a subbundle where $KZ$ acts strongly irreducibly with a positive exponent. For any bounded interval $[a,b] \subset \mathbb R$,  for all $\delta>0$, for all $L>0$ sufficiently large, 
there exists $\kappa_0:= \kappa_0(a,b,\delta, L)$ such that the following holds.
For  all $\kappa \in (0, \kappa_0)$,  there exists $T_3>0$ and $t_0>0$ such that,  for all $s_0 \in [a+e^{-2t},b-e^{-2t}]$, whenever 
\begin{itemize} 
\item  $g_th_{s_0}\omega \in \hat V_\epsilon:= V_\epsilon (\hat U(\delta, L))$,
\item $
\phi: [a,b] \to \mathcal{F}$ 
is a projectively $\kappa$-Lipschitz horocycle section at $\omega$, 
\item $t>t_0$ and  $T>T_3$
\end{itemize} 
 we have
\begin{multline} \label{eq:Zgrowing}
\text{\rm Leb}\Bigl(\{s \in [s_0-e^{-2t},s_0+e^{-2t}]:   |KZ(g_{t+T},h_s\omega)\phi(s)|  \\ <
e^{\mu(\epsilon, \delta) T}  \,|KZ(g_t,h_s\omega)\phi(s)| \} \Bigr) < \eta (\epsilon, \delta) e^{-2t}.
\end{multline}
\end{lem}
\begin{proof} 

Let $\delta>0$ and $L>0$ sufficiently large be fixed. 
By Corollary \ref{cor:all together} (iii) (and the fact that $\nu(\hat{U})>1-\delta$), since $g_th_{s_0} \omega \in \hat V_\epsilon$,  we have that, for all $T>T_2$, 
\begin{multline*} \int_0^T\int_{-1}^1 \chi_{\hat{U}}(g_\ell h_s g_th_{s_0}\omega)dsd\ell \\ = { e^{2t}}
\int_0^T\int_{-e^{-2t}}^{e^{-2t}} \chi_{\hat{U}}(g_\ell g_th_{s_0+s}\omega)dsd\ell>2T(1-\delta -\epsilon).
\end{multline*} 
That is,
 $$ \sum_{j=0}^{\lfloor \frac {T } {L}\rfloor-1}\int_0^L e^{2t} \int_{-e^{-2t}}^{e^{-2t}} \chi_{\hat{U}}(g_\ell g_{jL} g_th_{s_0+s}\omega)dsd\ell >2T-2T(\epsilon+\delta)-L.$$
From this we have that there exists $a \in [0,L]$ so that 
$$\sum_{j=0}^{\lfloor \frac {T} {L}\rfloor-1} e^{2t} \int_{-e^{-2t}}^{e^{-2t}} \chi_{\hat{U}}(g_{t+ a+jL} h_{s_0+s}\omega)ds  >\frac{2T-2T(\epsilon+\delta)-L}L.$$ 

It follows that except for a set of  $s \in [s_0-e^{-2t},s_0+e^{-2t}]$ with measure at most $ {e^{-2t}} \sqrt{2(\epsilon+\delta)+\frac L {T}} $
we have 
 $$\sum_{j=0}^{\lfloor\frac{T}L\rfloor -1} \chi_{\hat{U}}(g_{t+a+jL}h_s\omega)>\Big (1 -\sqrt{2(\epsilon+\delta)+\frac L {T}}\Big) \lfloor \frac{T}L\rfloor.$$ 
\ For  $t\geq t_0$ and for all  $s\in [s_0 -e^{-2t}, s_0 +e^{-2t}]$, let  
 $$
 \begin{aligned}
 J(s) := &\# \Big\{j\in \{ 0, \dots,
 \lfloor \frac {T} {L}\rfloor-1 \}:g_{t+a+jL}h_s\omega \in \hat{U} \\ & \text{ but } 
 |KZ(g_{t+a+(j+1)L},h_s\omega)\phi(s)|< e^{ {\lambda L}(1-\epsilon)}|KZ(g_{t+a+jL},h_s\omega)\phi(s)|\Big\}\,.
\end{aligned}
 $$
 We claim that 
 \begin{multline}
 \label{eq:in but not grow} 
 \text{\rm Leb}\Big(\Big\{s\in[s_0-e^{-2t},s_0+e^{-2t}]:  J(s) >2\sqrt{\delta} {\lfloor \frac{T}L\rfloor\Big\}\Big)}<2\sqrt{\delta}{e^{-2t}.}
\end{multline}
 In fact, for each $j$, let $\chi_j$ denote the characteristic function of the set 
 \begin{multline}
 \{s \in [s_0 -e^{-2t}, s_0 +e^{-2t}] : g_{t+a+jL}h_s\omega \in \hat{U} \\ \text{ but } |KZ(g_{t+a+(j+1) L},h_s\omega)\phi(s)| < e^{ {\lambda} L(1-\epsilon)}|KZ(g_{t+a+j L},h_s\omega)\phi(s)|\}.
 \end{multline}
 Since the interval $[a,b]$ is bounded, by Remark~\ref{rem:Lipschitz_1} there exists a constant $C_{a,b}>0$ such that if the section $\phi:[a,b] \to H^1(M, \mathbb R)$ is projectively $\kappa$-Lipschitz at $\omega$, then $\phi$ is projectively 
$C_{a,b} \kappa$-Lipschitz at $h_s\omega$, for all $s \in [a,b]$.  It then follows from  
Lemma~\ref{lem:stay_proj_lip}  that, for all $j\in \mathbb N$, the section $KZ(g_{t+a+j L})(\phi):[e^{2t}a, e^{2t}b] \to H^1(M, \mathbb R)$, is still  projectively $C_{a,b}\kappa$-Lipschitz at $g_{t+a+j L} h_s \omega$.  Thus, by Lemma \ref{lem:growth from E}, there exists $\kappa_0:= \kappa_0(a,b, \delta,L)$ such that, for all $\kappa \in (0, \kappa_0)$ 
if $t\geq t_0$, and $g_{t+a+j L}h_s\omega \in \hat{U} $, then 
 \begin{multline*}\label{eq:first bound}
\text{\rm Leb}( \{s' \in [s -e^{-2(t+a+j L)}, s +e^{-2(t+a+jL)}] :
 |KZ(g_{t+a+(j+1) L},h_s\omega)\phi(s)| \\ < e^{ {\lambda} L(1-\epsilon)}|KZ(g_{t+a+j L},h_s\omega)\phi(s)| \} )<
\delta e^{-2(t+a+jL)} \,.
 \end{multline*} 
Thus, 
 \begin{multline*}
\text{\rm Leb}(  \{s \in [s_0 -e^{-2t}, s_0 +e^{-2t}] : g_{t+a+jL}h_s\omega \in \hat{U} \text{ but } \\  |KZ(g_{t+a+(j+1) L},h_s\omega)\phi(s)| < e^{ {\lambda} L(1-\epsilon)}|KZ(g_{t+a+j L},h_s\omega)\phi(s)|\} )< 4 \delta e^{-2t}.
 \end{multline*}
 It then follows that
 $$
 \int_{-e^{-2t}}^{e^{-2t}}  \sum_{j=0}^{\lfloor \frac{T}L\rfloor -1 } \chi_j(s_0+s)  ds  \leq   \sum_{j=0}^{\lfloor \frac{T}L\rfloor -1 }  \int_{-e^{-2t}}^{e^{-2t}}  \chi_j (s_0+s) ds \leq 4\lfloor \frac{T}L\rfloor   \delta e^{-2t}\,.
 $$
which finally implies that 
 $$
\text{\rm Leb}  (\{ s\in   [s_0 -e^{-2t}, s_0 +e^{-2t}] :  \sum_{j=0}^{\lfloor \frac{T}L\rfloor -1 } \chi_j(s_0+s) > 2 \sqrt{\delta} \lfloor \frac{T}L\rfloor\})
\leq   2\sqrt{\delta} e^{-2t}\,,
 $$
 as claimed.
 
Choosing $T_3>0$ so large that $\frac L {T_3}<\epsilon+\delta$ gives a set of measure  at least $[1-(\sqrt{3(\epsilon+\delta)}+2\sqrt{\delta}) ] {e^{-2t}}$, where we have that, whenever $T>T_3$,  for at least $\frac {T}L({1-\sqrt{3(\epsilon+\delta)} -2\sqrt{\delta}})$ of our indices the cocycle grows by at least $e^{ {\lambda} L(1-\epsilon)}$.  By Lemma \ref{lem:KZbound}, the cocycle reduces {the Hodge norm} by multiplication times a factor larger than $e^{-L}$ on the remaining indices, thereby giving the bound. 
\end{proof}

\subsection{Some probabilistic results}
\label{subsec:largedev}
We collect below some well-known probabilistic results for the convenience of the reader. 
\begin{lem}\label{lem:comparison}
Let  $(\Omega,\mu)$ be a probability space and 
$F_i:(\Omega,\mu) \to \{0,1\}$ be a sequence of random variables such  that there exists $0<\rho<1$ so that for any $j$, the conditional probability that $F_j$ is $1$ given $F_1,\dots,F_{j-1}$ is at least $\rho$. Let $G_i:(\Omega,\mu) \to \{0,1\}$ be independent and so that $\mu(G_i^{-1}(1))=\rho$. 
 Then  for all $\ell$ and $r$,
$$\mu(\{\omega:\sum_{i=1}^\ell F_i(\omega)\leq r\})\leq \mu(\{\omega:\sum_{i=1}^\ell G_i(\omega)\leq r\}).$$
\end{lem}

By standard large deviations results we obtain:

\begin{cor} \label{cor:standard large deviations} Let  $(\Omega,\mu)$ be a probability space and 
$F_j:(\Omega,\mu) \to \{0,1\}$ be a sequence of random variables such  that there exists $0<\rho<1$ so that for any $j$, the conditional probability that $F_j$ is $1$ given $F_1,\dots,F_{j-1}$ is at least $\rho$. For all $\epsilon>0$ there exists $C_1,C_2>0$ so that 
$$\mu(\{\omega:\sum_{j=1}^\ell F_j(\omega)\leq (\rho-\epsilon)\ell\})\leq C_1e^{-C_2 \ell}.$$
\end{cor}

\begin{cor}\label{cor:large deviations}
Let  $(\Omega,\mu)$ be a probability space, $k \in \mathbb{N}$ and 
$F_j:(\Omega,\mu) \to \{0,1\}$ be a sequence of random variables such  that there exists $0<\rho<1$ so that for any $j$, the conditional probability that $F_j$ is $1$ given $F_1,...,F_{j-k}$ is at least $\rho$. For all $\epsilon>0$ there exists $C_3,C_4>0$ so that 
$$\mu(\{\omega:\sum_{j=1}^\ell F_j(\omega)\leq (\rho-\epsilon)\ell\})\leq C_3e^{-C_4 \ell}.$$
\end{cor}
 \begin{proof}  Let us consider the $k$ sequences of random variables: $(F_{1+jk})_{j\in \mathbb N}$, $(F_{2+jk})_{j\in \mathbb N}, \dots, (F_{k+jk})_{j \in \mathbb N}$. For each $i\in \{1, \dots, k\}$ the sequence
$(F_{i+jk})_{j\in \mathbb N}$ satisfies the hypothesis of Corollary~\ref{cor:standard large deviations}. In fact,
by assumption the conditional probabability that $F_{i +jk}$ is $1$ given $F_1,\dots, F_{i+jk-k}$ is at least $\rho$.
Since $\{F_{i},\dots, F_{i +(j-1)k}\} \subset \{F_1,\dots, F_{i+jk-k}\}$, it follows that the conditional probability that $F_{i +jk}$ is $1$ given $F_{i+1},\dots, F_{i +(j-1)k}$ is also at least $\rho$.
By Corollary~\ref{cor:standard large deviations} we therefore have that for each $i \in \{1, \dots, k\}$
$$\mu(\{\omega:\sum_{j=1}^\ell F_{i+jk} (\omega)\leq (\rho-\epsilon)\ell\})\leq C_1(\epsilon)e^{-C_2(\epsilon) \ell}.$$
Finally we have
$$
\{\omega:\sum_{j=1}^\ell F_j (\omega)\leq (\rho-\epsilon)\ell\} \subset \bigcup_{i=1}^k 
\{\omega:\sum_{j=1} ^{ \lfloor \frac{\ell-i} k \rfloor} F_{i+jk} (\omega)\leq (\rho-\epsilon)\ell/k\} \,,
$$
hence there exists $\ell_0:=\ell_0(k,\epsilon)$ such that for $\ell \geq \ell_0$ we have
\begin{multline}
\mu(\{\omega:\sum_{j=1}^\ell F_j (\omega)\leq (\rho-\epsilon)\ell\})  \\ \leq 
\sum_{i=1}^k \mu(\{\omega:\sum_{j=1} ^{ \lfloor \frac{\ell-i} k \rfloor} F_{i+jk} (\omega)\leq (\rho-\epsilon/2) \lfloor \frac{\ell-i} k \rfloor\}) \\\leq kC_1(\epsilon/2) e^{-C_2(\epsilon/2) (\ell/k -2)}\,.
\end{multline}
Thus the estimate in the statement holds (for $\ell$ sufficiently large) with $C_3(\epsilon) 
= kC_1(\epsilon/2)  e^{2C_2(\epsilon/2)}$ and $C_4(\epsilon)= C_2(\epsilon/2)$. 
\end{proof}

\section{The key proposition}\label{sec:key prop}

We recall that our strategy, based on Veech's criterion for weak mixing (see  \S\S \, \ref{ssec:Veech_criterion} below),  consists in establishing, on a large measure set of parameters,  growth  of the distance $\Vert \cdot \Vert_{\mathbb Z}$ from the integer lattice of the images of the cohomology classes of the line $\mathbb R \cdot [\Im(\omega)]$ under the Kontsevich--Zorich cocycle at arbitrarily large return times  of the Teichm\"uller orbit $\{g_t\omega \vert t>0\}$ to a given compact subset $\mathcal K$ of the moduli space. The following statement is the key step in our argument.

\begin{prop}\label{prop:better} Let $\kappa_0$ and $\mathcal{F}\subset \hat{\mathcal{F}}$ be as in Lemma \ref{lem:Zgrowth from V}. There exist  constants $\sigma>0,\tau>0$ so that for all large enough compact sets $\mathcal{K}$, there exists $\gamma_{\mathcal K}>0$ such that all horocycle sections $\phi:[-1, 1] \to \mathcal F \subset \FF$  at $\omega$, under the conditions that, for some 
$s_0 \in [-1,1]$ we have
\begin{enumerate}[label=(\alph*)]
\item $e^t $ large enough and  $\gamma \in (0, \gamma_{\mathcal K})$ ;  \item \label{cond:in cpct 2} $g_t h_{s_0} \omega \in \mathcal{K}$, 
\item \label{cond:lip transport} the section $\phi_{t,s_0} \colon [-1,1] \to H^1(M,\mathbb{R})$ defined by 
\begin{equation}
\label{eq:phi_t_s_0}
\phi_{t,s_0}(s)=KZ(g_t, h_{s_0 + s e^{-2t}}\omega) \phi(s_0  + se^{-2t})
\end{equation}
is $\gamma\kappa$-Lipschitz at $h_{s_0}\omega$,  with
$\kappa\in (0,\min \{1,\kappa_0\})$,  
\item\label{cond:max size} $ \max_{s\in [-1,1]}\Vert \phi_{t,s_0}(s)\Vert_{\mathbb Z}= \gamma$,
\end{enumerate}
there exists a set $\CalS_{t,s_0} \subset [s_0-e^{-2t}, s_0+ e^{-2t}]$ \footnote{Note that these are the parameters that \ref{cond:lip transport} refers to.}  of Lebesgue measure \begin{equation}\label{eq:meas lost}\text{\rm Leb} (\CalS_{t,s_0} )>2(1-\gamma^\sigma) e^{-2t}
\end{equation}
such that, for each $s \in \CalS_{t,s_0}$, there exists $\ell:=\ell (s)$ 
such that
\begin{enumerate}[label=(\Alph*)]
\item\label{conc:time} 
$ \frac 1 2\vert  \log \gamma \vert \leq \ell  \leq  \frac{3}{4}\vert\log \gamma\vert$;
\item\label{conc:in cpct}  $g_{t+\ell }h_{s} \omega \in \mathcal{K}$; 
 \item\label{conc:growth} $  \max_{s'\in [-1,1]} \Vert \phi_{t+\ell,s}(s')  \Vert_{\mathbb{Z}}\geq
 \Vert \phi_{t+\ell,s}(0)  \Vert_{\mathbb{Z}} >e^{\tau \ell} (1-\kappa)  \gamma$;

\end{enumerate}
\end{prop}

 The proposition formally implies the following corollary, with different choices of constants. Although the constants in the corollary may be different from those in the proposition, we denote them with the same symbol for ease of notation\footnote{In particular, the bounds on the number of subbundles and on the angle between them alter the estimates in formulas \eqref{eq:meas lost} and \ref{conc:growth} by multiplicative constants. However, given any $\sigma',\tau'$ strictly smaller than the minimum over the possible subbundles $\mathcal{F}$ of the constants $\sigma,\tau$ given in the previous proposition, if $\gamma'_{\mathcal{K}}$ is small enough, we can absorb the multiplicative constants. The basic fact is that if $c,\epsilon, C$  are fixed and $0<x<1$ is small enough then $Cx^{c+\epsilon}<x^c<\frac 1 C x^{c-\epsilon}$.}.
\begin{cor}\label{cor:better}
Let $\kappa_0$ be as in Lemma \ref{lem:Zgrowth from V}.   Let $\tilde{\mathcal{F}}$ be a direct sum of continuously varying strongly irreducible subbundles, each with a positive exponent. There exist  constants $\sigma>0,\tau>0$ so that for all large enough compact sets $\mathcal{K}$, there exists $\gamma_{\mathcal K}>0$ such that all horocycle sections $\phi:[-1, 1] \to \FF$  at $\omega$, under the conditions that, for some 
$s_0 \in [-1,1]$ we have
\begin{enumerate}[label=(\alph*)]
\item $e^t $ large enough and  $\gamma \in (0, \gamma_{\mathcal K})$ ;  \item \label{cond:in cpct 2} $g_t h_{s_0} \omega \in \mathcal{K}$, 
\item \label{cond:lip transport 2} the section $\phi_{t,s_0} \colon [-1,1] \to H^1(M,\mathbb{R})$ defined by 
\begin{equation}
\label{eq:phi_t_s_0 2}
\phi_{t,s_0}(s)=KZ(g_t, h_{s_0 + s e^{-2t}}\omega) \phi(s_0  + se^{-2t})
\end{equation}
is $\gamma\kappa$-Lipschitz at $h_{s_0}\omega$,  with
$\kappa\in (0,\min \{1,\kappa_0\})$,  
\item\label{cond:max size 2} $ \max_{s\in [-1,1]}\Vert \phi_{t,s_0}(s)\Vert_{\mathbb Z}= \gamma$,
\end{enumerate}
there exists a set $\CalS_{t,s_0} \subset [s_0-e^{-2t}, s_0+ e^{-2t}]$   of Lebesgue measure \begin{equation}\label{eq:meas lost 2}\text{\rm Leb} (\CalS_{t,s_0} )>2(1-\gamma^\sigma) e^{-2t}
\end{equation}
such that, for each $s \in \CalS_{t,s_0}$, there exists $\ell:=\ell (s)$ such that
\begin{enumerate}[label=(\Alph*)]
\item\label{conc:time 2} 
$ \frac 1 2\vert  \log \gamma \vert \leq \ell  \leq  \frac{3}{4}\vert\log \gamma\vert$;
\item\label{conc:in cpct}  $g_{t+\ell }h_{s} \omega \in \mathcal{K}$; 
 \item\label{conc:growth 2} $  \max_{s'\in [-1,1]} \Vert \phi_{t+\ell,s}(s')  \Vert_{\mathbb{Z}}\geq
 \Vert \phi_{t+\ell,s}(0)  \Vert_{\mathbb{Z}} >e^{\tau \ell} (1-\kappa)  \gamma$;
\end{enumerate}
\end{cor}
Note that by our assumption on $\ell(s_0)$, Conclusion \ref{conc:growth} implies that there exists $0<\rho<1$ so that : for all $s \in \CalS_{t,s_0}$ and for $\ell=\ell (s)$, we have 
\begin{multline}\label{eq:growth different form} 
 \Vert KZ(g_{\ell},g_th_{s}\omega)\phi(s)\Vert_{\mathbb{Z}}> (1-\kappa)\frac{\gamma^{\rho}}{\gamma} \Vert KZ(g_t, h_{s}\omega)\phi(s)\Vert_{\mathbb{Z}}\\
 \geq (1-\kappa)^2(\Vert KZ(g_t, h_{s}\omega)\phi(s)\Vert_{\mathbb{Z}})^\rho.
\end{multline}

Before beginning the proof of Proposition~\ref{prop:better}, we reduce lower estimates on the distance to the integer lattice to lower bounds on the norm of vectors in the cohomology bundle over a fixed compact set.

\begin{lem} \label{lem:dist_integers} For any compact subset $\mathcal K$ of the moduli space
there exists a constant $\gamma'_{\mathcal K}>0$ such that for $0<\gamma <\gamma'_{\mathcal K}$ the following holds. 
Let $\phi: [-1, 1] \to H^1(M, \mathbb R)$ be a $\gamma$-Lipschitz horocycle section at $\omega$. If  $g_{\ell}  \omega \in \mathcal K$ and 
$\max_{s\in [-1,1]}  \Vert \phi(s)\Vert_{\mathbb Z }\leq  \gamma$, then there exists a parallel section $z(s) \in H^1_{h_s \omega} (M,\mathbb Z)$ such that, for all $\ell \leq 3\vert \log \gamma\vert /4$ and for all  $s \in  [-1, 1]$ we have
 $$
 \Vert KZ(g_{\ell}, h_{ s}\omega) \phi(s)\Vert_{\mathbb Z}   = 
 \vert  KZ(g_{\ell}, h_{ s}\omega) (\phi(s) -z(s)) \vert \,.
 $$
\end{lem} 
\begin{proof}
We remark that over any compact set $\mathcal K$ the Hodge length of the shortest vector  of the integer lattice $H^1(M, \mathbb Z)$ has a positive minimum $\delta_{\mathcal K}>0$.  
 Since, by hypothesis, $\max \Vert  \phi(s)\Vert_{\mathbb Z } \leq \gamma$ it follows that 
there exists $z \in H^1_{\omega}(M, \mathbb Z)$  such that
 $$
 \vert \phi(0) -z  \vert =  \max_{s\in [-1,1]} \Vert  \phi(s)\Vert_{\mathbb Z } \leq   \gamma\,.
 $$
 Let $z(s) \in H^1_{h_{s}\omega}(M, \mathbb Z)$  denote the section given by the parallel transport of $z
 \in H^1_{h_{s_0}\omega}(M, \mathbb Z)$, that is, the section such that $\pi_\omega(s) z(s)=z$ for all 
 $s\in [-1,1]$. Similarly, let $\bar \phi(s) \in H^1_{h_{s}\omega}(M, \mathbb R)$ denote
 the parallel transport of the vector $\phi(0)$, that is, the section  $\bar \phi(s)$ such that
 $\pi_\omega(s) \bar \phi(s) = \phi(0)$, for all $s\in [-1,1]$. 
 
 Since  for all $s \in  [-1, 1]$ the vector
  $ \bar \phi(s) -z(s)$ is obtained from the vector $\phi(s_0) -z$  
  by parallel transport along a horocycle of length at most $2$,  for all $s \in  [-1, 1]$ we have
 $$
 \vert  \bar \phi(s) -z(s)   \vert  \leq  10 \gamma\,,
 $$
 and since by hypothesis $\phi$ is a $\gamma$-Lipschitz section
 it follows that 
 $$
 \vert \phi(s) -\bar \phi(s)  \vert  \leq 10 \gamma \,.
 $$
 We remark that since $\ell \leq \vert  \log \gamma\vert /2$ we have that $\Vert  KZ (g_\ell, \omega) \Vert \leq   \gamma^{-3/4}$.  Hence
$$
 \vert KZ(g_{\ell}, h_{s}\omega) (\phi(s) -z(s))  \vert  \leq  20 \gamma^{1/4}    \,.
 $$
Thus it suffices to choose $\gamma'_{\mathcal K}$ so that $(\gamma'_{\mathcal K})^{1/4}< \delta_{\mathcal K}/40$.
\end{proof}

As a consequence, under the hypotheses of the above lemma, a lower bound on $\vert KZ(g_{\ell}, h_{s}\omega)
(\phi(s) -z(s) )\vert $ is equivalent  to a lower bound on $\Vert KZ(g_{\ell}, h_{s}\omega) \phi(s) \Vert_{\mathbb Z} $
and up to replacing $\phi(s)$ with $\phi(s) -z(s)$ we can estimate $\vert KZ(g_{\ell}, h_{s}\omega) \phi(s) \vert$
from below.

Proposition \ref{prop:better} follows from the next lemma, whose proof we defer until after the proposition's proof.

For any $ \epsilon, \delta>0$, let $\eta:=\eta(\epsilon, \delta)>0$ and $\mu:=\mu(\epsilon, \delta)>0$ be the constants
defined in formula \eqref{eq:eta/kappa}. 

\begin{lem} \label{lem:key step}  Let $T_0>0$ be sufficiently large and so that $e^{T_0} \in \mathbb{N}$
and let  $T > T_0$.  Let $\kappa \in (0,\kappa_0)$ and let $\psi: [-1,1] \to H^1(M, \mathbb R)$ be a horocycle section 
at $\omega \in \mathcal K$ such that $\psi$ is projectively $\kappa $-Lipschitz. 
For all $j \in \{0, \dots,  \lfloor \frac T {T_0} \rfloor-1\}$, for $s\in [-1,1]$,  let 
$$
\Phi_{jT_0}(s) := KZ(g_{jT_0}, h_s\omega)\psi (s)\,.
$$
 If $\epsilon, \delta>0$ are small enough there exists $\sigma'>0$  such that 
 \begin{multline}\label{eq:better lem eq}
\text{\rm Leb}\Big(\{s\in [-1,1]:  \# \{j \in \{0,\dots, \lfloor  \frac T {T_0} \rfloor-1\}:   \\
\vert \Phi_{(j+1)T_0}(s)  \vert  \geq 
e^{(\mu-\epsilon) T_0}  \vert \Phi_{jT_0}(s) \vert \}
\\ >(1-16\eta)  \lfloor \frac T {T_0} \rfloor \}\Big)>2(1-e^{-\sigma' T})\,.
\end{multline}
\end{lem}

\begin{rem} Let $\mathcal{F}\subset \FF$ be an equivariant subbundle, where the Kontsevich-Zorich cocycle acts strongly irreducibly with a positive exponent, and let $\lambda$ be the largest such exponent. Let $v:[-1,1]\to \mathcal{F}$ be a parallel horocycle section at $\omega$. By the above lemma and Lemma \ref{lem:KZbound} one has that for any $\epsilon'>0$ there exists $ c>0$ and $T_1\in \mathbb{R}$ so that for all $T>T_1$ we have 
$$\lambda(\{s\in [-1,1]:|KZ(g_{ T },h_s\omega)v(s)|<e^{(\lambda-\epsilon') T}\})<e^{-cT}.$$
More generally, one can prove such a  large deviation result for any Lipschitz horocycle section $\psi:[-1,1]\to \mathcal{F}$. Indeed by Proposition~\ref{prop:in good integer} (applying it to a function $f_\emptyset$) there exists a compact set $\mathcal{K}$ and a constant $c$ so that the measure of the set of $s\in [-1,1]$ so that $|\{0\leq t\leq T:g_th_s \omega \notin \mathcal{K}\}|<cT$ decays exponentially with $T$. (Indeed one can choose $\rho$ so big that if $g_t\omega \in \mathfrak{C}_{\rho_0}$ then $g_{t+\tau}\omega \in \mathfrak{C}_\rho=\mathcal{K}$ for all $0\leq \tau \leq t_0$.)  By  Lemma  \ref{lem:stay_proj_lip}, the transported horocycle section $KZ(g_T)(\psi)$ at such an $s\in [-1,1]$ is 
$\kappa_T$-projectively Lipschitz with a constant $\kappa_T$ which decays exponentially with $T$.  
\end{rem}

\begin{proof}[Proof of Proposition \ref{prop:better} assuming~Lemma \ref{lem:key step}]   For all $c, \gamma >0$,  
 let  $\mathcal B:= \mathcal B(c, \gamma)$ denote the set of $s \in [-1,1]$ 
such that for all $\ell \in [\frac 1 2 |\log \gamma|,(\frac 1 2 +c)|\log \gamma|]$  we have   $g_{\ell} h_sg_{t}h_{s_0} \omega \not\in \mathcal{K}$.
We first prove that for any $c>0$, and for $\gamma>0$ sufficiently small, the Lebesgue measure of  $\mathcal B$  is polynomially small in $\gamma>0$. 

To prove an upper bound on the measure of $\mathcal B$ we apply Corollary~\ref{cor:ath} with $S=\frac 1 2 |\log \gamma|$ and  $T:=c|\log \gamma|$ sufficiently large to derive  that there exists $\zeta\in (0,1)$ such that
$$\text{Leb}(\mathcal B) \leq \zeta^{c |\log \gamma|} =  \gamma^{c  |\log \zeta|} .$$
 
Let $\psi=\phi_{t,s_0}$. 
By the hypotheses of the Proposition the horocycle section $\phi_{t,s_0}$ at $g_th_{s_0}\omega$,  defined in formula~\eqref{eq:phi_t_s_0}, satisfies the hypotheses of
Lemma~\ref{lem:key step}.   Indeed, by the hypotheses (c) and (d) of the  Proposition, the  section $\phi_{t,s_0}$  is $\kappa$-projectively Lipschitz at $h_{s_0}\omega$ with  $\kappa \in (0, \kappa_0)$.   

Let $T_0>0$ be sufficiently large and so that $e^{T_0} \in \mathbb{N}$ and let $\gamma >0$ such that $T:=\frac 1 2 \vert \log \gamma\vert > T_0$.   As in 
Lemma \ref{lem:key step}, we partition the interval $[0, \frac 1 2 \vert \log \gamma\vert]$  into $R:= \lfloor  \frac {\vert \log \gamma\vert} 2    \frac 1 {T_0} \rfloor$
equal intervals of length $T_0>0$ and an additional interval of length at most $T_0$.   
Let $\mathcal G$ be the set of formula~\eqref{eq:better lem eq} in the statement of the Lemma~
\ref{lem:key step} applied to the horocycle section $\phi_{t,s_0}$, and let $\CalS_{t,s_0}$ be defined as 
$$\mathcal S_{t,s_0}:= \{ s_0 + s' e^{-2t} \vert s' \in \mathcal G\setminus \mathcal B\}\,.$$
Lemma \ref{lem:key step} and the previous paragraph establishes  the measure lower bound of formula~\eqref{eq:meas lost}, for any $\sigma <  \min (\sigma', c  |\log \zeta|)$.

Let us  assume that $s \in \CalS_{t,s_0} \subset [s_0-e^{-2t}, s_0+ e^{2t}]$, hence by definition
$e^{2t} (s-s_0) \not \in \mathcal B$. By the above definition of the set $\mathcal B$, there  exists $$\ell_s \in [\frac 1 2 |\log \gamma|,(\frac 1 2 +c)|\log \gamma|]$$ 
so that $ g_{t+\ell_s} h_s \omega = g_{\ell_s} h_{e^{2t}(s-s_0)} g_t h_{s_0} \omega \in \mathcal{K}$. 

Clearly conditions~\ref{conc:time} and~\ref{conc:in cpct} are satisfied for $\ell:= \ell_s$. 
  It remains to verify Condition~\ref{conc:growth}. 
Since $e^{2t} (s-s_0)\in \mathcal G$, by formula~\eqref{eq:better lem eq} and Lemma~\ref{lem:KZbound} there exists  $\mu':= \mu-16\eta-\epsilon~>~0$ such that, for all $s\in [-1,1]$, 
$$
\begin{aligned}
 | KZ(&g_{RT_0}, g_th_{s}\omega)\phi(s)  | = | KZ(g_{RT_0}, h_{e^{2t} (s-s_0) }  g_th_{s_0}\omega)\phi_{t,s_0}(e^{2t} (s-s_0))  | \\ & = | \Phi_{RT_0}(e^{2t} (s-s_0)) |    \geq   e^{\mu' T_0 R}  |\phi_{t,s_0}(e^{2t} (s-s_0))  | =  e^{\mu' T_0 R}  |\phi(s)  |   \,.
\end{aligned} 
$$
and because the Hodge norm can change by a factor of at most $ e^{\pm (c |\log \gamma| +T_0)}$ from time 
$RT_0= \lfloor \frac{\vert \log \gamma\vert}2 \frac 1 {T_0} \rfloor  T_0  $   to time $\ell_s\leq (\frac 1 2  +c) |\log \gamma|$, for $c>0$ sufficiently small there exists $\tau >0$ such that
$$
| KZ(g_{\ell}, g_th_{s }\omega) \phi(s) |    \geq    e ^{ \tau \ell_s}    |  \phi_{t,s_0} (e^{2t}(s-s_0))| \,.
$$
Since, by our assumption on $s$,  $g_{\ell}  h_{e^{2t}(s-s_0)} g_t h_{s_0} \omega \in \mathcal K$,  by \ref{cond:lip transport} the horocycle section $\phi_{t,s_0}$ is $\gamma$--Lipschitz at $h_{e^{2t}(s-s_0)} g_t h_{s_0} \omega$   and  by \ref{cond:max size} 
\begin{equation}
\label{eq:max=gamma}
\max_{s'\in [-1, 1]} \Vert     \phi_{t,s_0}(s')   
\Vert_{\mathbb Z} =\gamma\,,
\end{equation} 
 by Lemma~\ref{lem:dist_integers}, there exists a parallel section $z_{t,s_0}: [-1,1] \to H^1(M, \mathbb Z)$
such that $z_{t,s_0}(s') \in H^1_{h_{s'} g_th_{s_0}\omega}(M, \mathbb Z)$ with
\begin{multline}
\Vert  KZ(g_{\ell}, h_{s'} g_th_{s_0}\omega) \phi_{t,s_0}(s')   \Vert_{\mathbb Z} \\ =
\vert KZ(g_{\ell}, h_{s'} g_t h_{s_0 }\omega)  (\phi_{t,s_0}(s')  -z_{t,s_0}(s')) \vert \,.
\end{multline}
In fact,  we can apply Lemma~\ref{lem:dist_integers} to the section $\phi_{t,s_0}$ to get a parallel integer section $z_{t,s_0}: [-1,1]\to  H^1(M, \mathbb Z)$ at  $g_t h_{s_0} \omega$ such that the above identity holds.
 By applying the above argument to the curve $\phi_{t,s_0}-z_{t,s_0}$ we therefore conclude that, for 
 $s\in \CalS_{t,s_0}$, 
$$
\Vert \phi_{t+\ell, s} (0)   \Vert_{\mathbb Z}=  \Vert KZ(g_{\ell}, g_t h_{ s }\omega) \phi_{t,s_0}(s)  (s )\Vert_{\mathbb Z}    \geq    
 e ^{ \tau \ell_s}   \Vert   \phi_{t,s_0}(s)  \Vert_{\mathbb Z}\,.
$$
Finally since by hypothesis the section $\phi _{t,s_0}$
is $\kappa\gamma$-Lipschitz and by formula~\eqref{eq:max=gamma}, we have the lower bound
$$
\Vert \phi_{t+\ell, s} (0)   \Vert_{\mathbb Z} \geq \min_{s'\in [-1,1]} \Vert   \phi_{t,s_0} (s') \Vert_{\mathbb Z} \geq
(1-\kappa) \gamma\,,
$$
which completes the argument.
\end{proof}

\begin{proof}[Proof of Lemma \ref{lem:key step}]
Let us recall that by definition, for $s\in [-1,1]$, 
$$
\Phi_{jT_0}(s) := KZ(g_{jT_0}, h_s\omega)\phi(s)\,.
$$
We say that the pair $(j,s)$  is   \emph{good} if  $g_{jT_0}h_s\omega\in V_\epsilon. $
 Observe that if $(j,s)$  is good then, since by hypothesis the horocycle section $\phi$ is projectively $\kappa$-Lipschitz, by Lemma \ref{lem:Zgrowth from V} for sufficiently large $T_0>0$, we have

 \begin{multline}\label{eq:useful growth} 
\text{\rm Leb} \Big(\Big\{s'\in [s-e^{-2jT_0 },s +e^{-2jT_0)}]: |\Phi_{(j+1)T_0, s_0}(s') | \\ <  e^{\mu T_0} |\Phi_{jT_0, s_0}(s') )| \Big\}\Big)< \eta e^{-2jT_0}. 
\end{multline}

We now wish to use this estimate and Corollary \ref{cor:large deviations} to complete the proof of the lemma. To satisfy the assumptions of Corollary \ref{cor:large deviations} we define a sequence of  nested partitions
{of the interval $[-1,1]$:} 
 $$\mathcal{P}^{(j)}=\{[ i e^{-2j {T_0}} , (i+1)e^{-2j {T_0}}]\}_{i=-  e^{2 j {T_0}}  }^{ e^{2 j {T_0}}  -1}.$$ 
   
 Let $P^{(j)}(s)$ be the unique element of $\mathcal{P}^{(j)}$ such that $s \in P^{(j)}(s)$  and 
{ $P^{(j)}_i=
 [ i e^{-2 j {T_0}},(i+1)e^{-2 j {T_0}}]$.}
 
 We claim that by the remark preceding the statement of Lemma~\ref{lem:key step}, 
 if  $g_{j{T_0}}h_s\omega \in \mathcal{K}$ then, for any $s'\in P^{(j+1)}(s)$, 
\begin{equation}\label{eq:same size}
(1+ 2e^{-2T_0})^{-1} \leq \frac{|\Phi_{jT_0}(s)|} {|\Phi_{jT_0}(s')|} \leq (1+ 2e^{-2T_0}) \,.
\end{equation} 
In fact, since  the Teichm\"uller distance between $g_{j{T_0}}h_{s}\omega$
and $g_{j{T_0}}h_{s'}\omega$ is at most $e^{-2T_0}$, there exists $g\in SL(2, \mathbb R)$ 
 at a  hyperbolic distance at most $e^{-2T_0}$  from the identity in $SL(2, \mathbb R)/SO(2)$  such that, for any cohomology class $v \in H^1_{h_{s'} \omega}(M, \mathbb R)$, we have
$$
KZ( g_{j{T_0}}, h_{s}\omega)(\pi_{h_s\omega}(v))= KZ(g, g_{j{T_0}} h_{s'}\omega)   KZ( g_{j{T_0}},h_{s'}\omega)(v) \,.
$$
By Lemma~\ref{lem:KZbound}  it follows that
$$
\vert KZ( g_{j{T_0}},h_{s}\omega)\pi_{h_s\omega}(v)\vert \leq \exp (e^{-T_0}) 
\vert KZ( g_{j{T_0}}, h_{s'}\omega)(v) \vert\,,
$$
so that for $v= \phi(s')$ we derive the estimate
$$
 \frac{|KZ( g_{j{T_0}},h_{s}\omega)\pi_{h_{s}\omega} (\phi(s'))|} {|KZ( g_{j {T_0}}, h_{{s'}}\omega)\phi(s')|}  \leq   \exp (e^{-2T_0})\,.
$$
Let $K>0$ denote the Lipschitz constant  of $\phi$ at $\omega$.  Since the section $\phi$ is $\kappa$-projectively Lipschitz at $\omega$, there exists a constant 
$C>0$ such that it is $CK$-Lipschitz and $C\kappa$-projectively Lipschitz at $h_s\omega$.  By Lemma~\ref{lem:KZbound}, we then have 
\begin{multline}
|KZ(g_{j {T_0}},h_{s'}\omega)\phi(s')| \geq  e^{-jT_0}\vert \phi(s')  \vert \\ \geq e^{-jT_0}(\vert \phi(s) \vert - C K ) \geq e^{-jT_0}K (\kappa^{-1} -  C )    \,.
\end{multline}
Since $\Vert KZ(g_{j{T_0}},h_{s}\omega)\Vert \leq e^{jT_0}$ and $\vert s-s' \vert \leq   e^{-2(j+1) T_0}$ for  $s'\in P^{(j+1)}(s)$, we also have
$$
|KZ(g_{j T_0},h_{s}\omega)(\phi (s) -\pi_{ h_{s}\omega} (\phi (s')))|  \leq  CK  e^{-2(j+1) T_0}\,,
$$
so that we have derived the estimate
$$
 \frac{|KZ(g_{j T_0},h_{s}\omega)(\phi (s) -\pi_{ h_{s}\omega} (\phi (s')))| }{|KZ(g_{j T_0},h_{s'}\omega)\phi(s')|} \leq  \frac{  C\kappa e^{-2T_0} } {1-C\kappa  }   \,.
$$
The upper bound in formula~\eqref{eq:same size} then follows from the above estimates, for $\kappa\in (0,\kappa_0)$ and for $T_0>0$ sufficiently large. The lower bound follows by symmetry, hence the claim is proved.

Now  assume $T_0$ is large enough so that $e^{-\epsilon T_0}<(1+e^{-2T_0})^{-1}$ and  let
$$B_j=\{s:  \exists s'  \in {P}^{(j)}(s)\text{ with } \frac{ |\Phi_{(j+1)T_0}(s')|}{   | \Phi_{jT_0}(s') |   }  \leq e^{(\mu-\epsilon)T_0} \leq  e^{ {\mu T_0}}(1+2e^{-2T_0})^{-1} \}\,, $$
let $\mathcal{I}_j(s)=\{k<j-1:s \in B_k\}$ and  
$$
\mathcal{G}_j=\{s: \exists s' \in {P}^{(j+1)}(s)\text{ with }g_{jT_0}h_{s'}\omega 
 \in V_\epsilon\}.
$$
Observe that  \eqref{eq:useful growth} and \eqref{eq:same size} give 

\begin{equation}\label{eq:cond good} 
\begin{aligned} \text{\rm Leb}(\{s \notin B_j: &s\in \mathcal{G}_j \text{ and } \mathcal{I}_j(s)=\vec{v}\}) \\ &>(1-2\eta)\text{\rm Leb}(\{s\in \mathcal{G}_j: \mathcal{I}_j(s)=\vec{v}\}).
\end{aligned} 
\end{equation}
Indeed, by definition $s,s'\in P^{(j)}_k$ for some $k$ implies $\mathcal{I}_j(s)=\mathcal{I}_j(s')$ and \eqref{eq:useful growth} and \eqref{eq:same size} gives the measure estimate conditioned to $s\in \mathcal{G}_j$ (which is equivalent to $s'\in \mathcal{G}_j$). 

This lower bound on the conditional probability provides the large deviations estimate via Corollary \ref{cor:large deviations}. 
Indeed, by applying the corollary to the sequence of random variables $F_j$ equal, for all $j\in \mathbb N$, to the characteristic functions of the sets of  $s \not\in B_j \cap \mathcal{G}_j$, we have that in the complement of a set of measure exponentially small  in $\frac{T}{T_0}$, 
\begin{equation}\label{eq:plenty in}
\#\{0\leq j \leq  \lfloor \frac{T}{T_0}\rfloor :s \notin B_j \,\text{ or } \,s\notin \mathcal{G}_j  \} >(1-8 \eta)  \lfloor \frac{T}{T_0}\rfloor 
\end{equation}
 (for all large enough  $T>T_0$). 
This is because by formula~\eqref{eq:cond good} the sequence of random variables $\{F_j\}$ defined above satisfies the assumptions of Corollary \ref{cor:large deviations}.   
Consider the set of $s\in [-1,1]$ such that 
  \begin{equation}\label{eq:in at integers}
  \vert  \{ j \in \{0,\dots, \lfloor \frac{T}{T_0}\rfloor \} : g_{j T_0} h_s \omega \in V_\epsilon\}\vert > (1-4\eta)  \lfloor \frac{T}{T_0}\rfloor .
  \end{equation}
Note that by~Corollary \ref{cor:all together} \ref{c:large dev} (which we may apply if $T_0$ is at least $t_0$) 
 the complement of this set has that its measure decays exponentially with  $\lfloor \frac{T}{T_0}\rfloor $. 
 The intersection of the sets given by \eqref{eq:plenty in} and \eqref{eq:in at integers} satisfies the desired conditions, hence the proof of  the Proposition is complete. 

\end{proof}

\section{Proof of Theorem \ref{thm:main}}\label{sec:main proof}
\subsection{The Veech criterion} 
\label{ssec:Veech_criterion}
To prove Theorem \ref{thm:main} we need a condition to rule out the flow $F^t_{r_\theta\omega}$ having 
$ \alpha \in \mathbb R$ as an eigenvalue.  This is Lemma~\ref{lem:v crit} below, which is essentially due to Veech. We remark that the second named author has an alternate version of the criterion, where one has a weaker assumption (only one time in the compact set instead of three) and obtains (essentially\footnote{The result of the second author assumes that the eigenvalue has no continuous eigenfunction; the case when the eigenvalue has a continuous eigenfunction has an easier proof following Veech's argument.}) the same conclusion \cite[Theorem 4.12]{FoSurv}.}
Following Veech \cite[\S 7]{veech metric 1} we have the following criterion for a translation flow $F^t_{r_\theta\omega}$ to not have $ \alpha$ as an eigenvalue. 
There is a $c>0$ and a sequence of transversals $\{J_i\}$ so that 
\begin{itemize} 
\item $|J_i|\to 0$.
\item $F_{r_\theta}^s(J_i)$ are disjoint intervals for all $0\leq s\leq \frac c {\ell'(J_i)}$, where $\ell'(J_i)$ denotes the length of the  projection of $J_i$ in the direction ${\theta}$. 
\item If $T$ is the IET given by the first return to $J_i$ and $I_1,\dots,I_d$ are the intervals that define $T$ then $|I_a|>c|I_b|$ for all $a,b$. 
\end{itemize}
Let $r_i$ be the vector of return times of $F^t_{r_\theta \omega}$ to $J_i$. If $\underset{i \to \infty}{\limsup }\,  \|r_i\alpha\|_{\mathbb{Z}}\neq 0$, then $\alpha$ is not an eigenvalue of $F^t_{r_\theta\omega}$. We have the following trivial consequence of this:

\begin{lem}\label{lem:v crit}(Veech) Let $\alpha$  be a nontrivial eigenvalue of $F^t_{\omega}$. For every compact set $\mathcal{K}$ there exists $\ell_{\mathcal{K}}$ so that, if there exist $\ell \geq \ell_{\mathcal{K}}$ and 
a diverging sequence  $(t_i) \subset \mathbb{R}^+$ with the property that
$g_{t_i}\omega \in \mathcal{K}$, $g_{t_i-\ell}\omega \in \mathcal{K}$ and  $g_{t_i+\ell}\omega \in \mathcal{K}$, then 
$$\underset{i \to \infty}{\lim} \|  KZ(g_{t_i},\omega)(\alpha\Im(\omega))\|_{\mathbb{Z}}=0. $$
\end{lem}
\begin{proof} For each compact set $\mathcal{K}$ there exist constants $a_{\mathcal{K}}, \, b_{\mathcal{K}}$ and $d_{\mathcal{K}}>0$ so that if $\omega' \in \mathcal{K}$ then we can write $\omega'$ as a  zippered rectangles  over an interval $I$ of length at most $a_{\mathcal{K}}$, in  the horizontal direction and one endpoint a singularity (that is, as a suspension with piece-wise constant roof function). Moreover the heights of the rectangles are at most $b_{\mathcal{K}}$. Lastly $d_{\mathcal{K}}$ is the length of the shortest  saddle connection on all $\omega'\in \mathcal K$.  Because there are singularities on each endpoint of the rectangles, if $u$ is the width of a rectangle, there is a saddle connection on the surfaces with holonomy $(u,r)$ with $r<b_{\mathcal{K}}$. So if $\ell>\log(\frac {b_{\mathcal{K}}}{2d_{\mathcal{K}}})$ and $g_\ell \omega' \in \mathcal{K}$ then $u\geq \frac 1 2 d_{\mathcal{K}}e^{-\ell}$. In fact, the holonomy of the same saddle connection on $g_\ell \omega$ is  $(e^{\ell} u, e^{-\ell} r)$ and for $\ell>\log(\frac {2b_{\mathcal K}}{d_{\mathcal K}})$ we have
$$
e^{-\ell} \vert r \vert  \leq  \frac {d_{\mathcal K}}{  2b_{\mathcal K}} b_{\mathcal K} = d_{\mathcal K}/2\,.
$$
Since $g_\ell \omega'  \in \mathcal K$ its shortest saddle connection has length at least $d_{\mathcal K}$,
hence we have $e^{\ell} \vert u \vert \geq  d_{\mathcal K}/2$, as claimed.

Likewise, because one endpoint of $I$ is a singularity, if $F^s_{\omega'}$ is discontinuous on $I$ then there is a saddle connection on $\omega'$ with holonomy $(a,v)$ with $|v|\leq |s|$ and $|a|<a_{\mathcal{K}}$. Thus if $\ell>\log(\frac{a_{\mathcal{K}}}{2d_{\mathcal{K}}})$   and $g_{-\ell} \omega' \in \mathcal{K}$, then $F^s_{\omega'}$ acts continuously on $I$ for all $-\frac 1 2 e^{-\ell}d_\mathcal K< s <\frac 1 2 e^{-\ell}d_\mathcal{K}$. In fact,  the saddle connection on $g_{-\ell} \omega'$  has holonomy $(e^{-\ell}a,e^{\ell}v)$. For $\ell>\log(\frac{2a_\mathcal{K}}{d_\mathcal{K}})$ we have
$$
e^{-\ell} \vert a \vert  \leq   \frac{ d_\mathcal K}{2 a_{\mathcal{K}}   } a_{\mathcal{K}} = d_\mathcal K/2\,,
$$ 
hence $e^{\ell}\vert v \vert \geq d_\mathcal K/2$ and so $\vert s\vert \geq  d_\mathcal K/2 e^{-\ell}$, as claimed.

We choose $\ell_{\mathcal{K}}=\max\{\log(\frac{a_{\mathcal{K}}}{2d_{\mathcal{K}}}),\log(\frac{b_{\mathcal{K}}}{2d_{\mathcal{K}}})\}$.

 Now, if $\omega'=g_{t_i}\omega$, by $g_{-t_i}$ we can transport $I$ back to $\omega$ and obtain an interval $I'$ so that $|I'|=\ell'(I')\leq e^{-t}  a_{\mathcal{K}}$.   Observe that  the second and third bullet points of the Veech criterion still  hold since they state properties which are equivariant with respect to the  action of the Teichm\"uller flow $g_t$.  We now need to understand the return times to $I'$. Note that they are the imaginary parts of a basis for (absolute) homology of  the surface and are all of bounded size when transported to $g_{t_i}\omega$ (because $g_{t_i}\omega=\omega' \in \mathcal{K}$). Thus the return time vector is given by $KZ(g_t,\omega)\Im(\omega)$ and the lemma follows. 
\end{proof}

\subsection{ Eliminating the stable subbundle of the $SL(2,\mathbb{R})$-bundle and isometric subbundles}\label{sec:stab and iso}

Recall the decomposition of $r_\theta$ given by formula~\eqref{eq:NAN}.
Since the horocycle ${\hat h}_{t}$ fixes the real part of holomorphic differentials, that is, it fixes the leaves of the stable foliation of the Teichm\"uller flow, we have that $g_{\log \cos \theta} h_{-\tan\theta} \omega$ belongs to the same stable manifold as  $r_\theta \omega$, hence  $h_{-\tan\theta} \omega$ and $r_\theta \omega$ have the same vertical
foliation. The vertical flow of $h_{-\tan\theta} \omega$ equals
  the vertical flow of $r_\theta \omega$  after a linear reparametrization  by multiplication times $\cos \theta$.  It follows that if $\alpha$ is a (fixed) eigenvalue for the vertical flow of 
of $r_\theta \omega$, then $\alpha\sec(\theta)$ is an eigenvalue for the vertical flow of  $h_{-\tan\theta} \omega$,
thus by the Veech criterion (Lemma \ref{lem:v crit}) we have that if $\mathcal{K}$ is a compact set, $\ell\geq\ell_{\mathcal{K}}$,  $g_{t_i}h_s\omega,g_{t_i-\ell}h_s \omega$ and $g_{t_i+\ell}h_s\omega \in \mathcal{K}$ with $t_i \to \infty$ then
$$\underset{i \to \infty}{\lim} \|KZ(g_{t_i},h_s\omega)\alpha  \sec(\theta) [\Im (h_{-\tan\theta}  \omega)] \|_{\mathbb{Z}}\to 0.$$
We conclude that if $\alpha$ is a fixed eigenvalue for the vertical flow of $r_\theta \omega$ for a positive measure set of $\theta \in \mathbb T$, then $ \alpha \sec (\arctan s) $ is an eigenvalue for the vertical flow of $h_s\omega$, for a positive measure set of $s \in \mathbb R$, which implies  that for a positive measure set of $s \in \mathbb{R}$ we have 
that if $g_{t_i}h_s\omega,g_{t_i-\ell}h_s \omega$ and $g_{t_i+\ell}h_s\omega \in \mathcal{K}$ with $t_i \to \infty$ then 
$$\underset{i \to \infty}{\lim} \|KZ(g_{t_i},h_s\omega) \psi(s)] \|_{\mathbb{Z}} \to 0$$ where $\psi(s)$ is the horocycle section 
$$
\psi( s) :=   \alpha \sec (\arctan s) ) [ \Im (h_s\omega))]  = \alpha \sec (\arctan s) ) [ \Im (\omega))] .
$$

\begin{lem} \label{lem:transv} The section $\psi$  has range in $SL(2,\mathbb{R})$ subspace at 
$\omega$ and  it is transverse to any integer translate of the stable subbundle of the $SL(2,\mathbb{R})$ subbundle.   
\end{lem}
\begin{proof}  The stable stable subspace of the $SL(2,\mathbb{R})$  subspace at 
$h_s\omega$ is generated by the cohomology class $[\Re (h_s\omega)]\in H^1_{h_s\omega}(M, \mathbb R)$.
For any $z \in H^1_{\omega} (M, \mathbb Z)$ we consider the equation  (with $a \in \mathbb R$)
$$
\psi(s) - z = a \Re(h_s\omega) = a (\Re (\omega) + s \Im(\omega)) \,.
$$
By definition of $\psi(s)$ we derive
$$
\begin{aligned}
z \wedge \Im(\omega) &=- a \Re(\omega) \wedge \Im(\omega) \,   \\ 
 z \wedge \Re(\omega) &=  (as  -\alpha \sec (\arctan s) )  \Re(\omega) \wedge \Im(\omega)\,.
\end{aligned} 
$$
Thus the coefficient $a \in \mathbb R$ is uniquely determined (given $z \in H^1_{\omega} (M, \mathbb Z)$)
by the first equation, hence the second equation has a unique solution.
\end{proof}

\begin{prop} \label{prop:central}  For every compact set $\mathcal K$ there exists a constant $\gamma''_{\mathcal K}>0$ such that the following holds. 
If there exist $t_0>0$  and an integer class $z \in H^1_{g_{t_0}\omega} (M, \mathbb Z)$ such that  
 \begin{itemize}
 \item  $g_{t_0} \omega \in \mathcal K$ and $\{t>0\vert g_t\omega \in \mathcal K\}$ has 
 upper density  at least $\frac 3 4 $;
\item $\vert KZ(g_{t_0}, \omega) (\alpha [\Im(\omega)] ) -z  \vert \in (0, \gamma''_{\mathcal K})$,\footnote{Note that this condition cannot hold if $\alpha=0$.}
\item $KZ(g_{t_0}, \omega) (\alpha [\Im(\omega)] ) -z$ belongs to an isometric
subbundle $E_0$  for the Kontsevich--Zorich cocycle,
 \end{itemize}
then $\alpha$ is not an eigenvalue for the vertical  flow $F^t_\omega$ of $\omega$. 
\end{prop} 

\begin{proof}[Proof of Proposition~\ref{prop:central}]   Let $\gamma''_{\mathcal K}$ be so small that the minimum Hodge distance between points of
the lattice $H^1_\omega (M, \mathbb Z)$  for $\omega \in \mathcal K$ is at least $10\gamma''_{\mathcal K}$.
By the isometry property of the Kontsevich--Zorich cocycle on $E_0$, we have that for all $t\geq t_0$, 
\begin{multline}
\Vert KZ(g_{t_0} , \omega) (\alpha [\Im(\omega)] ) \Vert_{\mathbb Z} = \vert KZ(g_{t_0}, \omega) (\alpha [\Im(\omega)] ) - z\vert \\ = \vert KZ(g_t, \omega) (\alpha [\Im(\omega)] ) - KZ(g_{t-t_0}, g_{t_0} \omega)(z)\vert  \leq \gamma''_{\mathcal K}\,,
\end{multline}
hence, whenever $g_{t} \omega  \in \mathcal K$, since $KZ(g_{t-t_0}, g_{t_0} \omega)(z) \in H^1_{g_t\omega} (M, \mathbb Z)$, by the choice of the constant $\gamma_{\mathcal K}>0$, we have
$$
\Vert KZ(g_t , \omega) (\alpha [\Im(\omega)] ) \Vert_{\mathbb Z}= \Vert KZ(g_{t_0} , \omega) (\alpha [\Im(\omega)] ) \Vert_{\mathbb Z} \not =0.
$$
Since by hypothesis that the set $\{t>0:g_t\omega \in \mathcal{K}\}$ has upper density~$\frac 3 4 $, it can be proved (by an elementary argument left to the reader) that for any $\ell_0>0$  there exist $\ell >\ell_0$ and a diverging sequence $(t_i)$ such that  $g_{t_i}\omega$, $g_{t_i-\ell}\omega$ and $g_{t_i+\ell}\omega \in \mathcal{K}$ for all $i\in \mathbb N$. Hence, we may apply the Veech criterion (Lemma~\ref{lem:v crit}) and $\alpha$ is not an eigenvalue, as stated.
\end{proof}

\subsection{Proof of Theorem \ref{thm:main}}
Before completing the proof, we recall that, if the Kontsevich-Zorich cocycle does not act isometrically on a subbundle, then   
Corollary \ref{cor:better} can be applied to it.  
\begin{thm} \label{thm:big quote} (c.f. \cite{EskMir}, \cite{FMZ}, \cite{Fi17}) The Hodge bundle (hence the symplectic orthogonal of the $SL(2,\mathbb R)$ subbundle) splits as a sum of an $SL(2,\mathbb R)$-invariant continuous  (in fact, algebraic) subbundle on which the Kontsevich--Zorich cocycle acts isometrically and a finite number of $SL(2,\mathbb R)$-invariant continuous (algebraic)  irreducible subbundles on each of which the cocycle has a strictly positive exponent. After passing to a finite cover, the irreducible components of the decomposition can be taken to be strongly irreducible, in the sense that they do not admit a measurable almost invariant splitting.
\end{thm} 
\begin{proof} By \cite[Theorem A.6]{EskMir} the Kontsevich--Zorich cocycle on Hodge bundle is semi-simple, in the sense that it can be split as a direct sum of $SL(2,\mathbb R)$-invariant measurable subbundles, and each  irreducible component is either isotropic or symplectic. By \cite[Theorem A.5]{EskMir} if the subbundle is isotropic then all of the exponents are zero. By \cite[Theorem A.4]{EskMir}, if all the exponents are zero then the cocycle acts isometrically. It follows that the Hodge bundle can be split as 
a sum of an isometric component, equal to the sum of all isotropic irreducible components, and a finite sum of symplectic irreducible components with at least one non-zero exponent. Each symplectic
non-isometric, component has both a positive and a negative exponent (which are opposite of each other). 
After passing to a measurable finite cover we can assume that each irreducible component is strongly irreducible. By \cite[Theorems 1.4 and 1.5]{Fi17} the splitting the Hodge bundle into irreducible components is continuous, in fact real algebraic, hence the irreducible components are strongly irreducible after passing to an algebraic finite cover. 
\end{proof}

\begin{cor}\label{cor:big quote} 
Let $\phi$ be a horocycle section. For almost every $s\in \mathbb R$, one of the following two (mutually exclusive) possibilities holds
\begin{enumerate}
\item\label{central} $\phi(s)$ {belongs to a continuous} $SL(2,\mathbb{R})$-equivariant subbundle where the Kontsevich-Zorich cocycle acts isometrically.
\item\label{unstable} $\phi(s)$ has a non-trivial projection onto a continuous $SL(2,\mathbb{R})$-equivariant subbundle where the Kontsevich-Zorich cocycle has a positive exponent and acts strongly irreducibly.
\end{enumerate}
\end{cor}

\begin{proof}[Proof of Theorem \ref{thm:main}]  We proceed with a proof by contradiction. Let $B \subset [-1,1]$ denote the set of points for which the conclusion of the theorem fails, and assume that the measure of $B$ is $\xi>0$.  Let $\mathcal{K}$ be a compact set so that  Lemma \ref{lem:in compact}  and Corollary \ref{cor:better} hold with $c=.999$ and  $\mu(Int(\mathcal{K}))>.999$. Let $a=\max\{t_0,\ell_{\mathcal{K}}\}$ where $t_0$ is as in Lemma \ref{lem:in compact} and  $\ell_{\mathcal{K}}$ is as in Lemma \ref{lem:v crit}. 
 We choose $\gamma_0 >0$  small enough so that Corollary \ref{cor:better} holds with the compact set $\mathcal{K}$, $\frac 1 2 |\log(\gamma_0 )|>T_0$, where $T_0$ is as in \S \ref{sec:key prop}, $\gamma_0\leq \gamma''_{\mathcal{K}}$ where $\gamma''_{\mathcal{K}}$ is as in Proposition \ref{prop:central}, and two additional smallness conditions below (so that the LHS in formula~\eqref{eq:beta} is small enough and the RHS in the estimate of formula~\eqref{cond:last} below is at most $\frac 1 2$). By Lemma \ref{lem:v crit} there exists $t_0>0$ so that 
  \begin{multline}
 \label{eq:weak stable density}
 \text{\rm Leb}(\{s \in [-1,1]:\text{for all } t\geq t_0   \text{ either }
  \Vert KZ(g_t,h_s\omega)\psi(s)\Vert_{\mathbb{Z}}<e^{-a-1}\gamma_0 \\
  \text{ or }g_t h_s\omega \notin \mathcal{K}, \text{ or }g_{t-a}h_s\omega  \notin \mathcal{K}, \text{ or }g_{t+a}h_s\omega \notin \mathcal{K}\})>.99\xi.
 \end{multline} Let $B'$ denote the set  in the LHS of formula~\eqref{eq:weak stable density}. By the Lebesgue density theorem, there exists $r_0>0$ so that 
  \begin{equation}\label{eq:density}
{\textrm{Leb}}(\{s\in B': {\textrm{Leb}}(B(s,r) \cap B' ) >1.99 r \text{ for all }r\leq r_0\})  > .98 \xi
 \end{equation}
 Let $B''$ be the set in the LHS of formula~\eqref{eq:density}. 
 We now choose $t\geq t_0$ and $ s\in  B''$ such that $h_{s}\omega$ is Birkhoff generic, which is a full measure condition by \cite[Theorem 1.1]{CE},  so that 
 \begin{itemize}
 \item $e^{-2t}<r_0$
 \item $g_th_{s}\omega \in \mathcal{K}$ 
 \item $\Vert KZ(g_t,h_s\omega)\psi(s )\Vert_\mathbb{Z}=\gamma< \gamma_0$.
 \end{itemize} 
 The first assumption is trivial.  The second and third assumption can be simultaneously satisfied by Lemma \ref{lem:v crit} and Birkhoff genericity. (Indeed $\mu(\mathcal{K})>.999>\frac 3 4 $ and so, for any $\ell$, there exist arbitrarilly large $t$ so that $g_{t-\ell}h_{s}\omega, g_th_\omega, g_{t+\ell}h_{s}\omega \in \mathcal{K}$.) 

 By the third, and last, of the above assumptions there exists an integer class $z \in H^1_{g_t h_s \omega} (M, \mathbb Z)$  such that 
 $$
 \vert KZ(g_t,h_s\omega)(\psi(s)-z) \vert = \gamma< \gamma_0\,. $$
We now show that if $\psi(s)-z$ is contained in an isometric subbundle, then for almost every $\theta\in S^1$, the flow in direction $\theta$ does not have $\alpha$ as an eigenvalue. Indeed, we are assuming all but the first assumption of Proposition \ref{prop:central} and that assumption holds for almost every $s\in [-1,1]$, by Birkhoff genericity. 

 Thus, up to the translation of $\psi$ by the integer vector $z$, that is, by considering the section $\psi -z$ instead of $\psi$, we can assume that $\psi$ has a non-zero projection  on a strongly irreducible subbundle with a
non-zero top Lyapunov exponent.
 
By Lemma~\ref{lem:transv}, the section $\psi$ is transverse to any integer translate of the stable subbundle of the $SL(2,\mathbb{R})$ subbundle, so clearly the projection to the stable subspace of the $SL(2,\mathbb{R})$ subbundle can be ignored. 
  
 To complete the proof we now separately consider the projection of $\phi$ to the unstable part of the 
 $SL(2, \mathbb R)$ subbundle and to the invariant subbundle $\hat{\mathcal{F}}^+\subset \hat{\mathcal{F}}$, given by the sum of all strongly irreducible subbundles with a positive exponent in the symplectic orthogonal $\hat{\mathcal{F}}$ 
 of the $SL(2, \mathbb R)$ subbundle. 
 
 We show that for each of these individually, at the time  $t'>t$ when the norm of the section has grown to size $\gamma_0$ under the action of the cocycle, the subset of $B(s, e^{-2t})$ of points which are not in the set in formula~\eqref{eq:weak stable density} has at least half the measure of $B(s,e^{-2t})$, a contradiction which will conclude our argument.
 First, we consider the projection onto the unstable part of the  $SL(2, \mathbb R)$-subbundle. By Lemma \ref{lem:in compact},  for any $t'>t_0+a$ we have that  
  \begin{multline}\label{eq:intwice}
 {\textrm{Leb}}(\{s'\in B(s,e^{-2t}): g_{t'}h_{s'}\omega, \, g_{t'-a}h_{s'}\omega \\  \text{ and }g_{t'+a}h_{s'}\omega \in \mathcal{K}\})>1.994e^{-2t}.
 \end{multline}
 Let $\delta>0$ be the norm of the projection of $KZ(g_t, h_s\omega) \psi(s)$   onto the unstable part of the $SL(2,\mathbb{R})$ subbundle. Let $\ell=\log(\frac {\gamma_0} \delta)+1$.
 Because the unstable part of the $SL(2,\mathbb{R})$ subbundle grows by a constant factor  $e^{t'}$ at time $t'>0$, we have our claim by applying \eqref{eq:intwice} at $t'=t+\ell$. 

It now suffices to prove the analogous result for $\hat{\mathcal{F}}=\hat{\mathcal{F}}^+\oplus \mathcal{I}$, where $\hat{\mathcal{F}}^+$ is, as above, the sum of all subbundles where the cocycle acts strongly irreducibly with a positive exponent and $\mathcal{I}$, where it acts isometrically. Such a decomposition is guaranteed by Theorem \ref{thm:big quote}.
Let $\phi$ be the projection of $\psi$ on $\hat{\mathcal{F}}^+$.  By choosing a possibly larger $t>0$, we can further assume that  
\begin{itemize}
\item $g_th_{s}\omega \in \mathcal{K}$,
\item the horocycle section $\phi_0:[-1,1]\to H^1(M,\mathbb{R})$ defined by 
$$\phi_0(s')=KZ(g_t,h_{s+e^{-2t}s'}\omega)\phi(s+e^{-2t}s')$$
 is $\kappa\gamma'$-Lipschitz at $h_{s} \omega $, 
 where $\kappa<\kappa_0$ as in Corollary  \ref{cor:better},
 \item $\max_{s'\in [-1, 1]} \Vert \phi_0 (s')\Vert_\mathbb{Z}=\gamma'< \gamma_0$. 
\end{itemize}
Indeed by Lemma~\ref{lem:stay_proj_lip}, for every compact set $\mathcal{K}'$, there exists $\upsilon>0$ so that to establish the second bullet point it suffices to show that 
$|\{0<t''<t:g_{t''}h_{s'}\omega \in \mathcal{K}'\}|>\upsilon >0$ for all $s'\in [s-e^{-2t},s+e^{-2t}]$. We choose $\mathcal{K}'$ to be the closure of a small neighborhood of $\mathcal{K}$. 
By the Birkhoff genericity of $h_s\omega$ with respect to the Teichm\"uller flow, there exists $t_1$ so that $|\{0<t''<t_1:g_{t''}h_{s}\omega \in \mathcal{K}\}|>\upsilon>0$. There exists $U$, a small neighborhood of $s$, so that for every $s'\in U$ and $0\leq t''\leq t_1$, 
$$g_{t''}h_{s'}\omega\in \mathcal{K}'\, \text{ whenever }\, g_{t''}h_s\omega \in \mathcal{K}.$$ The second bullet point follows by choosing $t\geq t_1$ so that ${s'\in [s-e^{-2t},s+e^{-2t}]\subset U.}$

We now iteratively apply Corollary~\ref{cor:better}.  For every $j\in \mathbb N$,  let $\mathcal S_{j} \subset [s-e^{-2t},s + e^{2t}]$ denote the set of points such that Corollary~\ref{cor:better} can be applied $j$-times. 
Inductively, by applying Corollary~\ref{cor:better} for $j-1$ times, for a point $s_0 \in \mathcal S_{j} $ we obtain real numbers $\ell_1(s_0),\dots,\ell_{j}(s_0)$. Let $L_j(s_0) := \ell_1(s_0)+\dots +\ell_j(s_0)$ and, for  all $s'\in [-1,1]$,  let
$$
\phi_{j}(s' ) :=  KZ (g_{t+L_{j}(s_{0})}, h_{s'} \omega) \phi(s')  \,,
$$
which restricts on the interval $[s_0-e^{-2\big(t+L_j(s_0)\big)},s_0+e^{-2\big(t+L_j(s_0)\big)} ]$ to a reparametrization of the section 
$\phi_{t+L_{j}(s_0), s_0}$ as in Corollary \ref{cor:better} \ref{conc:growth 2}. 
Since the section $\phi_{t+L_{j}(s_0),s_0}$ 
is $\gamma' \kappa$-Lipschitz and has maximum at least $\gamma'$ it follows that 
$$
\min \phi_{t+L_j(s_0),\, s_0} \geq (1-\kappa) \max \phi_{t+L_j(s_0), s_0}\,.
$$
Let then $(x_j)$ denote the sequence defined by recursion as follows: let $x_0 = \min \phi_0$ and let 
$$
x_{j+1} =(1-\kappa)^2 x_j^{\rho}.
$$
Notice that, for all $j\in \mathbb N$,
$$
x_j = (1-\kappa)^{2+ \dots + 2\rho^j} x_0^{\rho^j} \geq   (1-\kappa)^{2/(1-\rho)} { x_0  }^{\rho^j}\,.
$$
Observe that, if 
\begin{equation}
\label{eq:beta}
 \beta:=\log_{ x_0  }(\gamma_0)- 2(1-\rho)^{-1} \log_{x_0 }
  (1-\kappa) >0\,.
 \end{equation}  
 and $\upsilon =  \lceil \log_{\rho} (\beta)\rceil$ then, if $\gamma_0$ is small enough, we have

 \begin{equation}\label{eq:bigger than threshold} 
2\gamma_0^\rho > x_\upsilon\geq \gamma_0\,.
 \end{equation} 
Indeed, $x_0^{\rho^{\log_\rho(\beta)}}=\gamma_0(1-\kappa)^{-\frac{2}{1-\rho}}.$
 Now, by induction, if $s'\in \mathcal{S}_{j+1}$ and $\|KZ(g_{t+L_i( s' )},h_{s'}  \omega) \phi(s') \|_{\mathbb{Z}}<\gamma_{\mathcal{K}}$ for $i\leq j$ then 
 \begin{equation}\label{eq:bigger than x}
 \|KZ(g_{t+L_{j+1} (s' )},h_{s'}\omega) \phi(s') \|_{\mathbb{Z}}\geq x_{j+1}. 
 \end{equation} 
 Indeed, inductively $x_i \leq \|KZ(g_{t+L_i(s')},h_{s'}\omega)\phi(s')\|_{\mathbb{Z}}$ and so it follows by formula   
 \eqref{eq:growth different form}. To see this, we now check that, if $s' \in \mathcal{S}_j$, then we can apply Corollary~\ref{cor:better} at $h_{s'}\omega$ iff 
$\|KZ(g_{t+L_j(s')},h_{s'}\omega)\phi(s')\|$ is small enough. 
First, $e^{t+L_j(s')}>e^t$ and $g_{t+L_j(s')}h_{s'}\omega \in \mathcal{K}$ by the definition of $\mathcal{S}_j$. Second, the Lipschitz assumption follows because 
Lemma~\ref{lem:stay lip} gives that the Lipschitz constant gets smaller, while at the same time $\|KZ(g_{t+L_j(s')},h_{s'}\omega)\phi(s')\|\geq \|KZ(g_{t},h_{s'}\omega)\phi(s')\|$ and so the requirement on the Lipschitz constant gets laxer. So the second bullet point implies \ref{cond:lip transport 2}.
 
We now bound from below the measure of the set of these points to obtain a contradiction with \eqref{eq:density}, thereby completing the proof. By \eqref{eq:meas lost 2} we have that the the measure of the set of $s' \in [s-e^{-2t},s+e^{-2t}]$ so that $\|KZ(g_{t+L_{\min\{j:s' \notin S_j\}-1}(s')},h_{s'}\omega ) \phi(s') \|_{\mathbb{Z}}<\gamma_0$ is at most (since $\upsilon =  \lceil \log_{\rho} (\beta)\rceil$ and by formula \eqref{eq:beta})  
\begin{equation}\label{cond:last}
\sum_{j=0}^{\upsilon} (x_j)^{\sigma} \leq  \sum_{j=0}^{\upsilon} ({ x_0}^{\rho^j})^{\sigma}\le \sum_{i\ge -1}  (x_0  ^{\beta\sigma})^{\frac 1 {\rho^i}}
= \sum_{i\ge 0}  \left ( \frac{\gamma_0} {(1-\kappa)^{{2}/(1-\rho)  }}\right)^{\rho \sigma \frac 1 {\rho^i}} .
\end{equation}

Clearly this  bound  goes to 0 with $\gamma_0$,  and so if $\gamma_0$ is small enough we are left with at least a subset of conditional measure greater than $\frac 1 2 $. That is, for a set of $s'\in [s-e^{-2t},s + e^{-2t}]$ of conditional measure at least $\frac 1 2 $ we have that for each $s'$ in this set: 
 \begin{itemize}
\item there exists $\tau_{s'}\geq t$ so that  $\|KZ(g_{\tau_{s'}},h_{s'}\omega) \phi(s')\|_{\mathbb{Z}}\geq \gamma_0$\,,
\item $g_{\tau_{s'}}h_{s'}\omega \in \mathcal{K}$. 
 \end{itemize}
 By Lemma~\ref{lem:in compact} if $\gamma_0$ is sufficiently small  (so that in Corollary~\ref{cor:better} $\ell (s) \geq \vert \log \gamma_0 \vert/2 >t_0$, 
  hence $\tau_{s'} -t \geq t_0$) the subset of $s' \in [s-e^{-2t}, s+e^{-2t}]$  such that $g_{\tau_{s'}+a} h_{s'}\omega$  or $g_{\tau_{s'}+2a}h_{s'}\omega$ are not in $\mathcal{K}$ has conditional measure at 
 most~$1/9$.  By Lemma \ref{lem:KZbound} we have that 
 $$
 \|KZ(g_{\tau_{s'}},h_{s'}\omega) \phi(s')\|_{\mathbb{Z}}\geq \gamma_0 \Longrightarrow \|KZ(g_{\tau_{s'}+a},h_{s'}\omega) \phi(s')\|_{\mathbb{Z}}\geq e^{-a} \gamma_0\,.
 $$
 By the above lower bound, together with the condition that $g_{\tau_{s'}}h_{s'}\omega \in \mathcal{K}$,
 $g_{\tau_{s'}+a}h_{s'}\omega \in \mathcal{K}$ and $g_{\tau_{s'}+2a}h_{s'}\omega \in \mathcal{K}$, 
 we conclude that $s'\not\in B'$.  Since the conditional measure of such $s'\in [s-e^{-2t},s+e^{-2t}]$ is
 at least $1/2-1/9> 1/4$, this contradicts that $s \in B''$ (because $e^{-2t}<r_0$). 
 \end{proof}

\section{Proof of Theorem \ref{thm:billiard}}\label{sec:billiard}

\subsection{Setup}
Let $P$ be a polygon and $T_1(P)$ be the phase space for its billiard flow. Let $X(P)=T_1(P)\times T_1(P)$.

The phase space $T_1(P)$ is endowed with the Liouville measure  for the billiard flow, hence the space $X(P)$
can be endowed with the product measure. All integrals below will be taken with respect to the square of the Liouville measure on $X(P)$.

We also need to define $Lip_c(X(P))$ the space of c-Lipschitz functions on $X(P)$. For instance, we view $T_1(P)$ as a quotient of $P\times S^1$ (which is a metric space) and use the path metric to define a metric on $T_1(P)$.  

Theorem~\ref{thm:billiard} follows readily from the next lemma and the Baire category Theorem. The proof of this lemma will  be presented after the proof of Theorem \ref{thm:billiard} assuming the lemma.
\begin{lem} \label{lem:open} For every $\epsilon>0$ and $T>0$ we have 
\begin{equation}\label{eq:open}
\begin{aligned}
\{P:  \int_{X(P)} \Big\vert \frac{1}{T}\int_0^Tf(F^t(x,\theta),&F^t(y,\psi))dt-\int_{X(P)}f \Big\vert <\epsilon  \\   &\text{ for all }f\in \LipP\}
\end{aligned}
\end{equation}
is open.
\end{lem}
By Theorem \ref{thm:erg prod} we have

\begin{cor} \label{cor:dense} For every $\epsilon>0$ there exists $N_\epsilon$ such  that the following holds. If $P$ is a rational polygon with the property that the group of reflections about its sides (translated to the origin) contains a rotation by $\frac{2\pi} {M}$ with $M\geq N_\epsilon$, then there exists $T_0$ such that for all $T\geq T_0$ and $f \in \LipP$ we have
$$ \int_{X(P)}  \vert \frac{1}{T}\int_0^Tf(F^t(x,\theta),F^t(y,\psi))dt-\int_{X(P)}f|<\epsilon.$$
\end{cor}
\begin{proof} 
Consider the flat surface $M_P$ obtained by unfolding $P$. 
 By Theorem~\ref{thm:erg prod} we have that for almost every pair of directions the product flow is ergodic on $M_P\times M_P$. Considering these flows on  $T_1(P)$, by our assumption on the group of reflections, they 
equidistribute in the product of  the table cross $M\geq N_\epsilon$ evenly spaced copies of a discrete set  in $S^1$. The corollary follows.
\end{proof}
\begin{proof}[Proof of Theorem \ref{thm:billiard} assuming Lemma \ref{lem:open}]
Let $n$ be given and, for all $k\in \mathbb N$,  let $U_k$ be the set of polygons $P$ with $n\geq 3$ sides so that
there exists $T>0$ such that, for all $f \in \LipP$,  $$\int_{X(P)} |\frac{1}{T}\int_0^Tf(F^t(x,\theta),F^t(y,\psi))dt-\int_{X(P)}f|<\frac 1 k.$$
By Lemma \ref{lem:open}, $U_k$ is open for all $k\in \mathbb N$ as union (over $T>0$) of a family of open sets. By  Corollary \ref{cor:dense} it is dense for all $k\in \mathbb N$. So by the Baire Category theorem there exists $P\in \cap_{k}U_k$.  Because 1-Lipschitz functions have dense span in the set of continuous functions, the product of the billiard flow with itself on $X(P)$ is ergodic.  In fact, if $P \in \cap_{k}U_k$ we have that, for every $k\in \mathbb N$, there
exists $T_k>0$ such that for all $f \in \LipP$,  
$$\int_{X(P)} |\frac{1}{T_k}\int_0^{T_k}f(F^t(x,\theta),F^t(y,\psi))dt-\int_{X(P)}f|<\frac 1 k.$$
Ergodicity then follows (for instance) from the ergodic theorem. The ergodicity of the product billiard flow implies that the billiard flow is weakly mixing.
\end{proof}

\begin{defin} (Y.~Vorobets \cite{vorobets},  Definition 2.1) 
We say a polygon $Q$ is a \emph{$\delta$-perturbation} of $P$ if there exists a homeomorphism $\phi_{P,Q}:P \to Q$ that establishes a one-to one correspondence between the vertices of the two polygons and so that the distance between the corresponding vertices is at most $\delta$. 
\end{defin}

 For every polygon $P$ there exists $\delta_P>0$ such that for $\delta <\delta_P$, any $\delta$-perturbation
$Q$ of $P$ has a triangulation whose triangles are in a bijective correspondence with triangles of a triangulation of $P$. For instance,
one can triangulate $P$ by diagonals in an arbitrary manner. Then for $\delta < d(P)/2$, with $d(P)$ the minimum
of the smallest non-zero distance between vertices, for each two vertices of $P$ which can be joined by a diagonal inside $P$, the corresponding vertices of $Q$ can also be joined by a diagonal inside $P$.  Then there
exists a triangulation of $Q$ corresponding to the above triangulation of $P$ by diagonals.

In the next lemma we assume the homeomorphism $\phi_{P,Q}$ is as above, affine on the triangles in a (fixed) triangulation of $P$ and takes them to the corresponding triangles of $Q$.  
\begin{lem}\label{lem:lip in lip}For every table $P$ and for every $\epsilon>0$,  there exists~$\delta>0$ such that if  $Q$ is a $\delta$-perturbation of $P$, then $$Lip_{1-\epsilon}(X(Q))\subset \phi_{P,Q}^{-1}(\LipP)\,.$$
\end{lem}
This is straightforward.

\begin{lem}\label{lem:lip tolerance}
If $P$ is in the set described by \eqref{eq:open} then there exists $\epsilon'>0$ so that 
$$
\begin{aligned} \int_{X(P)} \Big\vert \frac{1}{T}\int_0^Tf(F^t(x,\theta),F^t(y,\psi))dt-&\int_{X(P)}f \Big\vert<\epsilon   \\
& \text{ for all }f\in Lip_{1+\epsilon'}(X(P)).
 \end{aligned} 
$$
\end{lem} 
\begin{proof} 
The table $P$ is compact and $\LipP$ is bounded and  equicontinuous, so by the Arzel\`a-Ascoli Theorem, $\LipP$ is compact in $\|\cdot\|_{\sup}$. Thus, there exists $0<\epsilon''<\frac 1 9$ so that 
\begin{equation}
\label{eq:eps2}
\begin{aligned}  \int_{X(P)} \Big\vert\frac{1}{T} \int_0^Tf(F^t(x,\theta),F^t(y,\psi))dt-&\int_{X(P)}f \Big\vert<\epsilon-\epsilon'' \,,   \\ &\text{ for all }f\in \LipP.
\end{aligned}
\end{equation}
Indeed, for each $f \in \LipP$ there exists $\epsilon_f$ so that 
$$\begin{aligned}  \int_{X(P)} \Big\vert\frac{1}{T} \int_0^Tf(F^t(x,\theta),F^t(y,\psi))dt-&\int_{X(P)}f \Big\vert<\epsilon-\epsilon_f.
\end{aligned}$$
Thus, there is an open neighborhood $U_f$ of $f$  that, for all $h \in U_f$,
$$\begin{aligned}  \int_{X(P)} \Big\vert\frac{1}{T} \int_0^Th(F^t(x,\theta),F^t(y,\psi))dt-&\int_{X(P)}h \Big\vert<\epsilon-\frac 1 2 \epsilon_f\,.
\end{aligned}$$
 The claim in formula \eqref{eq:eps2} follows by compactness.

Let then $\epsilon'=\frac 1 {9}(\frac 1 {1-\epsilon''}-1)$ and $h \in Lip_{1+\epsilon'}(X(P))$.
Now  from \eqref{eq:eps2} applied to $f=h/(1+ \epsilon')$ we derive the following bound:
$$\begin{aligned}  \int_{X(P)} \Big\vert\frac{1}{T} \int_0^Th(F^t(x,\theta),F^t(y,\psi))dt-&\int_{X(P)}h \Big\vert<(1+\epsilon')(\epsilon-\epsilon'')<\epsilon.
\end{aligned}$$
The last inequality uses that $\epsilon' \epsilon<(1+\epsilon')\epsilon''$. 
 \end{proof}

For any polygon $P$, let $\pi^{(1)} : P\times S^1 \to P$ and $\pi^{(2)} : P\times S^1 \to S^1$ denote the canonical projections.

\begin{prop}(Vorobets \cite[Proposition 2.3]{vorobets})  
Let $Q$ be a $\delta$-perturbation of $P$. For each $t>0$ there exists $B \subset P \times S^1$ dependent on the polygons $P,Q$, the map $\phi:= \phi_{P,Q}$ and $t$ and of measure at most $C_3(C_1t+C_2)^3 \delta$ such that for each $(x,\theta) \in P \times S^1 \setminus B$ and for each $\tau$, $0\leq \tau \leq t$, at least one of the following two possibilities hold:
\begin{enumerate}
\item The distance between $\pi_P^{(1)}F_p^t(x,v)$ and $\phi^{-1}(\pi_Q^{(1)}F_Q^t(\phi(x),v)$ is at most $C_4(C_1t+C_2)^2 \delta$ and the angle between the directions of $\pi^{(2)}F^t_P(x,v)$ and $\pi^{(2)}F^t_Q(\phi(x),v)$ is at most $C_5(C_1t+C_2)\delta$. 
\item The points $\pi_P^{(1)}F_p^t(x,v)$ and $\pi_Q^{(1)}F_Q^t(\phi(x),v)$  lie at a distance at most $C_6(C_1t+C_2)^2\delta$ from the boundaries of $P$ and $Q$ respectively.
\end{enumerate}
The constants $C_1,\dots,C_6$ are positive and depend only on $P$. 
\end{prop}
From this proposition, we derive:
\begin{cor}For any $f \in Lip_c(X(P))$, $T$ and $\epsilon>0$ there exists $\delta>0$ so that for any $Q$, a $\delta$-perturbation of $P$ we have 
\begin{multline}\label{eq:close wm averages}
\int_{X(Q)} \Big\vert \int_0^T f\circ \phi_{P,Q}^{-1}(F^t(x,\theta),F^t(y,\psi))dt  \\  -\int_0^Tf(F^t(x,\theta),F^t(y,\psi))dt \Big\vert <\epsilon. 
\end{multline}
\end{cor}
Note for fixed $c$ and $T$, $\delta$ can be chosen uniformly over $Lip_c(X(P))$

\begin{proof}[Proof of Lemma \ref{lem:open}] Let $\epsilon<\frac 1 4$.  It suffices to show that for any $P$ 
 which belongs to the set in formula~\eqref{eq:open}   there exists $\delta_P>0$ so that any $\delta_P$-perturbation also belongs to the set~\eqref{eq:open}. Because $P$ belongs to the set~\eqref{eq:open}, it satisfies the inequality in
~\eqref{eq:open} for $\epsilon-\epsilon'''$ for some $\epsilon'''>0$. We choose $\delta>0$ so that Lemma \ref{lem:lip in lip} implies that, if $Q$ is a $\delta$-perturbation of $P$, then $Lip_1(X(Q))\subset Lip_{1+\frac {\epsilon'''}9}(X(P))$, formula \eqref{eq:close wm averages} is satisfied with $\epsilon=\frac {\epsilon'''}9$
  and
$$
\Big\vert \int_{X(Q)} f\circ \pi_{P,Q}^{-1}-\int_{X(P)}f \Big\vert <\frac {\epsilon'''}9\,.
$$
It follows that any $\delta$-perturbation of $P$ satisfies \eqref{eq:open}. 
\end{proof}

\end{document}